\documentclass[preprint,pdftex,english,11pt]{imsart}
\usepackage{amsthm,amsmath,amssymb,amsfonts,epic,multirow,ulem}

\usepackage{thmtools}

\usepackage{tgbonum}

\usepackage[numbers]{natbib}
\RequirePackage{natbib}

\usepackage{wrapfig}
\usepackage[english]{babel}
 \usepackage[T1]{fontenc} 

 \usepackage{makeidx}
 \usepackage{fancyhdr}
\usepackage{epsf}
\usepackage[dvips]{epsfig}  
\usepackage{subfig}
 \usepackage{graphicx}
\usepackage[nice]{nicefrac}
\usepackage{dsfont}
\usepackage{calrsfs} 

\usepackage{pifont}
\usepackage{bm}

\usepackage{wasysym}
\usepackage{stmaryrd}
\usepackage{aeguill}
\usepackage{mathrsfs}
\usepackage{enumerate}
\usepackage[shortlabels]{enumitem}
\usepackage{enumitem}
\usepackage[dvipsnames,table,xcdraw]{xcolor}
\usepackage{cancel}
\usepackage{ulem}
\usepackage{caption}
\usepackage{tikz}

\usepackage{hyperref}
\definecolor{linkcolour}{rgb}{0,0.2,0.6}
\hypersetup{colorlinks,breaklinks,urlcolor=linkcolour, linkcolor=linkcolour, citecolor=linkcolour}

\usepackage{geometry}

\startlocaldefs

\newtheoremstyle{mytheoremstyle} 
        {\topsep}                    
        {\topsep}                    
        {\itshape\fontfamily{ppl}\selectfont}                   
        {}                           
        {\fontfamily{ppl}\selectfont\bfseries\color{black}}                   
        {.}                          
        {.5em}                       
        {}  
\theoremstyle{mytheoremstyle}

\newtheorem{theo}{Theorem}[section]

\newtheorem{lemm}[theo]{Lemma}

\newtheorem{Fact}[theo]{Fact}

\makeatletter
\renewenvironment{proof}[1][\proofname]{\par
  \pushQED{\qed}%
  \fontfamily{ppl} \topsep6\p@\@plus6\p@\relax
  \trivlist
  \item[\hskip\labelsep\itshape\bfseries#1\@addpunct{.}]\ignorespaces}{%
  \popQED\endtrivlist\@endpefalse
}

\def\R{\mathbb{R}}
\def \N{\mathbb{N}}

\def \P{\mathbb{P}} 

\def \E{\mathbb{E}} 

\newcommand{ \un }{\mathds{1}}

\def\T{\mathbb{T}}

\newenvironment{merci}{\textbf{Acknowledgments.}}{ }
\endlocaldefs

\newtheorem{Ass}{Assumption}
\newtheorem{remark}{Remark}

\renewcommand{\T}{\mathbb{T}}
\newcommand{\X}{\mathbb{X}}

\def \Eb{\mathbf{E}}
\def \Pb{\mathbf{P}}

\renewcommand{\P}{\mathbb{P}}

\newtheorem{postita}{Post-it}

\renewcommand{\P}{\mathbb{P}}

\def\T{\mathbb{T}}

\makeatletter
\newcommand*\bigcdot{\mathpalette\bigcdot@{.5}}
\newcommand*\bigcdot@[2]{\mathbin{\vcenter{\hbox{\scalebox{#2}{$\m@th#1\bullet$}}}}}
\makeatother

\newcommand{\VectCoord}[2]{#1^{(#2)}}

\newcommand{\sqrtBis}[1]{#1^{1/2}}

\newcommand{\Pbis}{\mathds{P}}
\newcommand{\Ebis}{\mathds{E}}

\newcommand{\tn}{\lfloor tn\rfloor}

\newcommand{\Mn}{\lfloor Mn\rfloor}
\newcommand{\an}{\lfloor an\rfloor}
\newcommand{\bn}{\lfloor bn\rfloor}
\newcommand{\PEsqrt}{\lfloor\sqrtBis{n}\rfloor}

\newcommand{\CardRoots}{B^{(n)}}

\DeclareMathAlphabet\mathbfcal{OMS}{cmsy}{b}{n}

\begin{document}

{\fontfamily{ppl}\selectfont

\begin{frontmatter}


\title{Genealogy in critical generations of a diffusive random walk in random environment on trees}

\author{\fnms{Alexis} \snm{Kagan}\ead[label=e2]{alexis.kagan@auckland.ac.nz}}
\address{Department of Statistics, University of Auckland, New Zealand. \printead{e2}} \vspace{0.5cm}

\runauthor{Kagan}


\runtitle{Genealogy in critical generations of a diffusive random walk in random environment on trees}

\begin{abstract}
We consider the range $\mathcal{R}^{(n)}$, the tree made up of visited vertices by a diffusive null-recurrent randomly biased walk $\X$ on a Galton-Watson tree $\T$ up to the $n$-th return time to its root and we consider the following genealogy problem: pick two vertices uniformly at random in a generation of order $n$ in the tree $\mathcal{R}^{(n)}$. Where does the coalescence occur? it turns out that the coalescence happens either in the recent past or in the remote past.

\end{abstract}



\end{frontmatter}


\section{Introduction}

\subsection{Randomly biased random walk on a Galton-Watson tree}\label{RBRWT}

Under a probability measure $\Pb$, we consider a Galton-Watson tree $\T$ rooted at $e$ with offspring distribution $N$. We assume $\Eb[N]>1$ so that $\T$ is a super-critical Galton-Watson that is $\Pb(\textrm{non-extinction of }\T)>0$ and we define $\Pb^*(\cdot):=\Pb(\cdot|\textrm{non-extinction of }\T)$, where $\Eb$ (resp. $\Eb^*$) denotes the expectation with respect to $\Pb$ (resp. $\Pb^*$). \\
For any vertex $x\in\T$, we denote by $|x|$ the generation of $x$, by $x_i$ its ancestor in generation $i\in\{0,\ldots,|x|\}$ and $x^*:=x_{|x|-1}$ stands for the parent of $x$. In particular, $x_0=e$ and $x_{|x|}=x$. For convenience, we add a parent $e^*$ to the root $e$. For any $x,y\in\T$, we write $x\leq y$ if $x$ is an ancestor of $y$ ($y$ is said to be a descendent of $x$) and $x<y$ if $x\leq y$ and $x\not=y$. We then write $\llbracket x_i,x\rrbracket:=\{x_j; j\in\{i,\ldots,|x|\}\}$. Finally, for any $x,y\in\T$, we denote by $x\land y$ the most recent common ancestor of $x$ and $y$, that is the ancestor $u$ of $x$ and $y$ such that $\max\{|z|;\; z\in\llbracket e,x\rrbracket\cap\llbracket e,y\rrbracket\}=|u|$. In the present paper, we restrict ourselves to the case $\#\{x\in\T;\; |x|=m\}<\infty$ for any integer $m\geq 0$. \\
We also consider the branching random random walk $\mathcal{E}:=(\T,(V(x);x\in\T))$ on the real line with the convention $(V(x);\; |x|=0)=\{V(e)\}=\{0\}$, that is a $\bigcup_{k\in\N}\R^k$-valued random variable ($\mathbb{R}^0$ only contains the sequence with length $0$) such that for any integers $0\leq m\leq n$, for any $|u|=m$ and $|x|=n$, if $V_u(x):=V(x)-V(u)$, then given $\mathcal{F}_m:=\sigma(\{z\in\T;\; |z|\leq m, V(z)\})$, the collection $\{(V_u(x);\; |x|=n,\; x\geq u);\; |u|=m\}$ is made up of\textrm{ i.i.d }copies of $(V(x);\; |x|=n-m)$, independent of $\mathcal{F}_m$. We will refer to $\mathcal{E}$ as the random environment. \\
Given a realization of the random environment $\mathcal{E}$, we define a $\T\cup\{e^*\}$-valued nearest-neighbour random walk $\X:=(X_j)_{j\in\N}$, reflected in $e^*$ whose transition probabilities are, under the quenched probabilities $\{\P^{\mathcal{E}}_z; z\in\T\cup\{e^*\}\}$: for any $x\in\T$, $\P^{\mathcal{E}}_x(X_0=x)=1$ and 
\begin{align*}
    p^{\mathcal{E}}(x,x^*)=\frac{e^{-V(x)}}{e^{-V(x)}+\sum_{y;\;y^*=x}e^{-V(y)}}\;\;\textrm{ and for any child }z\textrm{ of }x,\;\; p^{\mathcal{E}}(x,z)=\frac{e^{-V(z)}}{e^{-V(x)}+\sum_{y;\;y^*=x}e^{-V(y)}}.
\end{align*}
Otherwise, $p^{\mathcal{E}}(x,z)=0$ and $p^{\mathcal{E}}(e^*,e)=1$. Let $\P^{\mathcal{E}}:=\P^{\mathcal{E}}_e$, that is the quenched probability of $\X$ starting from the root $e$ and we finally define the following annealed probabilities
\begin{align*}
    \P(\cdot):=\Eb[\P^{\mathcal{E}}(\cdot)]\;\;\textrm{ and }\;\;\P^*(\cdot):=\Eb^*[\P^{\mathcal{E}}(\cdot)].
\end{align*} 
R. Lyons and R. Pemantle \cite{LyonPema} initiated the study of the randomly biased random walk $\X$. $V(x)=|x|\log\lambda$ for a some constant $\lambda>0$, the walk $\X$ is known as the $\lambda$-biased random walk on $\T\cup\{e^*\}$ and was first introduced by R. Lyons (see \cite{Lyons} and \cite{Lyons2}). The $\lambda$-biased random walk is transient (on the set of non-extinction) unless the bias is strong enough: if $\lambda\geq\Eb[N]$ then, $\Pb$-almost surely, $\X$ is recurrent (positive recurrent if $\lambda>\Eb[N]$). R. Lyons, R. Pemantle and Y. Peres (see \cite{LyonsRussellPemantle1} and \cite{LyonsRussellPemantle2}), later joined by G. Ben Arous, A. Fribergh, N. Gantert, A. Hammond \cite{BA_F_G_H} and E. Aïdékon \cite{AidekonSpeed} for example, studied the transient case and payed a particular attention to the speed $v_{\lambda}:=\lim_{n\to\infty}|X_n|/n\in[0,\infty)$ of the random walk.\\
When the bias is random, the behavior of $\X$ depends on the fluctuations of the following $\log$-Laplace transform: for any $t\geq 0$ 
\begin{align*}
    \psi(t):=\log\Eb\Big[\sum_{|x|=1}e^{-tV(x)}\Big], 
\end{align*}
and we assume that $\psi$ is finite is a neighborhood of $1$ and that $\psi'(1)$ exists. As stated by R. Lyons and R. Pemantle \cite{LyonPema}, if $\inf_{t\in[0,1]}\psi(t)$ is positive, then $\Pb^*$-almost surely, $\X$ is transient and we refer to the work of E. Aïdékon \cite{Aidekon2008} for this case. Otherwise, it is recurrent. More specifically, G. Faraud \cite{Faraud} proved that the random walk $\X$ is $\Pb$-almost surely positive recurrent either if $\inf_{t\in[0,1]}\psi(t)<0$ or if $\inf_{t\in[0,1]}\psi(t)=0$ and $\psi'(1)>0$. It is null recurrent if $\inf_{t\in[0,1]}\psi(t)=0$ and $\psi'(1)\leq 0$. In the null-recurrent case for which we will be interesting in, it is now very well known that the behavior of the random walk $\X$ differs drastically whether $\psi'(1)=0$ or not. Indeed, when $\inf_{t\in[0,1]}\psi(t)=0$ and $\psi'(1)=0$, also referred to as the boundary case for the underlying random environment $\mathcal{E}$, the largest generation reached by the walk $\X$ up to the time $n$ is of order $(\log n)^3$ (see \cite{HuShi10a} and \cite{HuShi10b}), but surprisingly, the right order for $|X_n|$ is $(\log n)^2$ (see \cite{AndDeb2} and \cite{HuShi15} for instance). When $\inf_{t\in[0,1]}\psi(t)=0$ and $\psi'(1)<0$ however, the random walk reached larger generations. Define
\begin{align}\label{DefKappa}
    \kappa:=\inf\{t>1;\; \psi(t)=0\},
\end{align}
and assume $\kappa\in(1,\infty)$. It as been proved that the sequences $((|X_{\tn}|/n^{1-1/(\kappa\land 2)};\; t\geq 0)_{n\geq 1}$ if $\kappa\not=2$ and $((|X_{\tn}|(\log n)^{1/2}/n^{1/2};\; t\geq 0))_{n\geq 1}$ if $\kappa=2$ both converge to an explicit and non-degenerate process, see \cite{AidRap} and \cite{deRaph1}. Therefore, the random walk is said to be sub-diffusive when $\kappa\in(1,2]$ and diffusive when $\kappa>2$. The present paper is dedicated to the latter case, that is

\begin{align}\label{DiffCase}
    \inf_{t\in[0,1]}\psi(t)=\psi(1)=0,\;\;\psi'(1)<0\;\;\textrm{ and }\;\;\kappa>2.
\end{align}
\noindent We require
\begin{Ass}\label{Assumption1}
there exists $\delta_1>0$ such that $\psi(t)<\infty$ for all $t\in[1-\delta_1,2+\delta_1]$ and 
\begin{align*}
    \Eb\Big[\Big(\sum_{|x|=1}e^{-V(x)}\Big)^{2+\delta_1}\Big]<\infty. 
\end{align*}
\end{Ass}
\noindent and 
\begin{Ass}\label{Assumption2}
    the distribution of the $\bigcup_{k\in\N}\R^k$-valued random variable $(V(x);\; |x|=1)$ is non-lattice.
\end{Ass}

\subsection{The range of the random walk}

\noindent Let $\tau^0=0$ and for any $j\geq 1$
\begin{align*}
    \tau^j:=\inf\{k>\tau^{j-1};\; X_{k-1}=e^*,\; X_k=e\},
\end{align*}
with the convention $\inf\varnothing=+\infty$. It is known (see \cite{AndDeb1}, \cite{Hu2017} and more recently Theorem 1.2 in \cite{AK23LocalTimes}) that under the Assumptions \ref{Assumption1} and \ref{Assumption2}, $\Pb^*$-almost surely, in law under $\P^{\mathcal{E}}$, the sequence of random variables $(\tau^n/n)_{n\geq 1}$ converges to $\mathfrak{c}(W_{\infty})^2\bm{\tau}_{-1}$ for some explicit constant $\mathfrak{c}>0$, where $\bm{\tau}_{-1}$ stands for the first hitting time of $-1$ by a standard Brownian motion and $W_{\infty}$ is the limit of the additive martingale $(\sum_{|x|=k}e^{-V(x)})_{k\geq 0}$. It is also known that still under the Assumptions \ref{Assumption1} and \ref{Assumption2}, $\Pb(W_{\infty}>0)>0$, see \cite{Biggins1977}, \cite{Lyons1997}, \cite{Liu1} or \cite{Alsmeyer_Iksanov} for instance. Moreover, it is claimed in \cite{Biggins1977} that $\Pb$-almost surely, the event $\{W_{\infty}>0\}$ coincides with the event of non extinction of the underlying Galton-Watson tree $\T$. In particular, $\Pb^*(W_{\infty}>0)=1$.

\vspace{0.2cm}

\noindent For any $x\in\T$ and $p\geq 1$, define the edge local time $\VectCoord{N}{p}_x$ by
\begin{align*}
    \VectCoord{N}{p}_x:=\sum_{j=1}^{\tau^p}\un_{\{X_{j-1}=x^*,\; X_j=x\}},
\end{align*}
the number of times the oriented edge $(x^*,x)$ has been visited by the random walk $\X$ up to $\tau^p$. Let us also introduce the range $\VectCoord{\mathcal{R}}{p}$ of $\X$
\begin{align*}
    \VectCoord{\mathcal{R}}{p}:=\big\{x\in\T;\; \VectCoord{N}{p}_x\geq 1\big\},
\end{align*}
the sub-tree of $\T$ made up of the vertices visited by the random walk $\X$ up to time $\tau^p$.
\begin{Fact}[Lemma 3.1, \cite{AidRap}]\label{FactGWMulti}
Under $\P$, for any $p\in\N^*$, $(\VectCoord{\mathcal{R}}{p},(\VectCoord{N}{p}_x;\; x\in\VectCoord{\mathcal{R}}{p}))$ is a multi-type Galton-Watson tree with initial type equal to $p$. Moreover, we have the following characterization: for any $x\in\T$, $x\not=e$, any $k_1,\ldots,k_{N(x)}\in\N$ and $k\geq 1$, on the event $\{N^{(1)}_x=k\}$
\begin{align*}
    \P^{\mathcal{E}}\Big(\bigcap_{i=1}^{N(x)}\{N_{x^i}^{(1)}=k_i\}\big|N_z^{(1)};\; z\leq x\Big)=\frac{(k-1+\sum_{i=1}^{N(x)}k_i)!}{(k-1)!k_1!\cdots k_{N(x)}!}\times p^{\mathcal{E}}(x,x^*)^{k}\prod_{i=1}^{N(x)}p^{\mathcal{E}}(x,x^i)^{k_i},
\end{align*}
and clearly, if $N_x^{(1)}=0$, then $N_y^{(1)}=0$ for all $y\geq x$.    
\end{Fact}

\subsection{The volume of the range in critical generations}
For any $p\geq 1$ and $k\geq 0$, introduce $\VectCoord{\mathcal{R}}{p}_{k}$, the generation $k$ of the tree $\VectCoord{\mathcal{R}}{p}$
\begin{align*}
    \VectCoord{\mathcal{R}}{p}_{k}:=\{x\in\VectCoord{\mathcal{R}}{p};\; |x|=k\}.
\end{align*}
We denote by $\VectCoord{R}{p}_{k}$ the cardinal of $\VectCoord{\mathcal{R}}{p}_{k}$. Note that under the Assumptions \ref{Assumption1} and \ref{Assumption2}, the sequence of random variables $(\max_{x\in\mathcal{R}^{(n)}}|x|/n)_{n\geq 1}$ converges in law to an explicit and positive random variable (see Fact 3.4 in \cite{AK24LocCrtiGen}). Hence, generations of order $n$ are called critical generations for the tree $\mathcal{R}^{(n)}$. \\
Let us now introduce a very important sequence of random variables. For any $p\geq 1$ and $k\geq 0$, define
\begin{align}\label{DefMultiMart}
    \VectCoord{Z}{p}_k:=\sum_{|x|=k}\VectCoord{N}{p}_x.
\end{align}
Let $\VectCoord{\mathcal{G}}{p}_k$ be the sigma-algebra generated by $\{x\in\T, |x|\leq k,N_x^{(p)}\}$. One can notice, thanks to Fact \ref{FactGWMulti} that $\E[\sum_{|x|=k}\VectCoord{N}{p}_x|\VectCoord{\mathcal{G}}{p}_{k-1}]=\E[\sum_{|x|=k-1}\sum_{z;z^*=x}\VectCoord{N}{p}_z|\VectCoord{\mathcal{G}}{p}_{k-1}]=\sum_{|x|=k-1}
\VectCoord{N}{p}_x$. Hence, $(\VectCoord{Z}{p}_k/p)_{k\geq 0}$ is, under the annealed probability $\P$ and for any $p\in\N^*$, a non-negative $(\VectCoord{\mathcal{G}}{p}_k)_{k\in\N}$-martingale such that $\E[\VectCoord{Z}{p}_k/p]=1$ for all $k\in\N$ and $\VectCoord{Z}{p}_0/p=1$. $(\VectCoord{Z}{p}_k/p)_{k\geq 0}$ is referred to as the \textit{multi-type additive martingale} in \cite{deRaph1}. We also define 
\begin{align}
    \VectCoord{L}{p}_k:=\sum_{j=1}^{\tau^p}\un_{\{|X_j|=k\}},
\end{align}
the local time of $(|X_j|)_{j\geq 0}$ at time $\tau^p$ and level $k$. There exists a very simple relation between $(\VectCoord{L}{p}_k)_{k\geq 0}$ and $(\VectCoord{Z}{p}_k)_{k\geq 0}$: for any $k\geq 0$, $\VectCoord{L}{p}_k=\VectCoord{Z}{p}_k+\VectCoord{Z}{p}_{k+1}$.

\vspace{0.3cm}

\noindent Before stating our next theorem, we need a few more notations and definitions. Introduce the constant $c_0:=\Eb[\sum_{x\not=y;\; |x|=|y|=1}e^{-V(x)-V(y)}]/(1-e^{\psi(2)})$ which is well defined by Assumption \ref{Assumption2}. Let $(\mathcal{Y}_a;\;a\geq 0)$ be a continuous-state branching process (CSBP) with branching mechanism $\lambda\mapsto c_0\lambda^{2}$ such that $\mathcal{Y}_0=1$, in the sense that for almost every environment $\mathcal{E}$, $(\mathcal{Y}_a;\; a\geq 0)$ is a real-valued Markov process starting from $1$ such that for any $\lambda\geq 0$ and any $0\leq a\leq b$
\begin{align}\label{LaplaceCSBP}
    \E^{\mathcal{E}}\Big[e^{-\lambda\mathcal{Y}_b}\big|\mathcal{Y}_a\Big]=\exp\Big(-\lambda\mathcal{Y}_a\big(1+c_0(b-a)\lambda\big)^{-1}\Big). 
\end{align} 
The random process $(\mathcal{Y}_a;\;a\geq 0)$ has continuous paths and is also referred to as the Feller diffusion, that is the unique strong solution of
\begin{align*}
    \mathrm{d}\mathcal{Y}_a=\big(2c_0\mathcal{Y}_a\big)^{1/2}\mathrm{d}B_a\;\;\textrm{ and }\;\;\mathcal{Y}_0=1,
\end{align*}
where $(B_a;\; a\geq 0)$ is a standard Brownian motion. \\
Let $(S_j-S_{j-1})_{j\in\N^*}$ be a sequence of\textrm{ i.i.d }real valued random variables such that $S_0=0$ and for any measurable, non-negative function $f:\R\to\R$
\begin{align}\label{LoiRW}
    \Eb[f(S_1)]=\Eb\Big[\sum_{|x|=1}e^{-V(x)}f(V(x))\Big].
\end{align}
Introduce the positive constant $\bm{c}_{\infty}:=\Eb[(\sum_{j\geq 0}e^{-S_j})^{-1}]$. Our first result is joint convergence for the range, the multi-type additive martingale and the local time, in generation $\an$ and up to time $\tau^n$.
\begin{theo}\label{ThMultiMar}
Assume that the Assumptions \ref{Assumption1} and \ref{Assumption2} hold. We have, in law under $\P$ for the space of càdlàg functions $D([0,\infty),\R^3)$
\begin{align*}
        \Big(\frac{1}{n}\VectCoord{R}{n}_{\an},\;\frac{1}{n}\VectCoord{Z}{n}_{\an},\;\frac{1}{n}\VectCoord{L}{n}_{\an};\;  a\geq 0\Big)\underset{n\to\infty}{\longrightarrow}\big(\bm{c}_{\infty}W_{\infty}\mathcal{Y}_{a/W_{\infty}},\;W_{\infty}\mathcal{Y}_{a/W_{\infty}},\;2W_{\infty}\mathcal{Y}_{a/W_{\infty}};\; a\geq 0\big).
    \end{align*}
\end{theo}
\noindent Hence, we refer to the generations of order $n$ as the critical generations for $\mathcal{R}^{(n)}$. Note that $\VectCoord{R}{n}_{\an}$, $\VectCoord{L}{n}_{\an}$  and $\VectCoord{L}{n}_{\an}$ are proportional as $n\to\infty$. Indeed, it turns out that the vertices in $\{|x|=\an;\; N^{(n)}_x=1\}$ give a contribution of the same order as the vertices in $\{|x|=\an;\; N^{(n)}_x>1\}$ to the set $\mathcal{R}^{(n)}_{\an}$. \\
Note that for any $m\geq 1$, under $\Pb$
\begin{align}\label{SmoothingTransform}
    W_{\infty}\overset{(\textrm{law})}{=}\sum_{|u|=m}e^{-V(u)}W^{(u)}_{\infty},
\end{align}
where $(W^{(u)}_{\infty};\; |u|=m)$ is a collection of\textrm{ i.i.d }copies of $W_{\infty}$ and independent of $((u,V(u));\; |u|=m)$. Hence, if we set $\mathcal{Z}_a(t):=t\mathcal{Y}_{a/t}$ for any $t>0$ and $\mathcal{Z}_a(0)=0$, then, by \eqref{LaplaceCSBP}, under $\P$
\begin{align}\label{NotSmoothingTransform}
    \mathcal{Z}_a(W_{\infty})\overset{(\textrm{law})}{=}\sum_{|u|=m}e^{-V(u)}\mathcal{Z}^{(u)}_{ae^{V(u)}}(W_{\infty}^{(u)}),
\end{align}
where for almost-every environment $\mathcal{E}$, $(\mathcal{Y}^{(u)};\; |u|=m)$ is a collection of\textrm{ i.i.d }copies under $\P^{\mathcal{E}}$ of $\mathcal{Y}$. But clearly, $(\mathcal{Z}^{(u)}_{ae^{V(u)}}(W_{\infty}^{(u)});\; |u|=m)$ is not a collection of\textrm{ i.i.d }copies of $\mathcal{Z}_a(W_{\infty})$. \\
Also, one can see that $\P^{\mathcal{E}}$-almost surely, $\lim_{t\to 0}\mathcal{Z}_a(t)=0$, thus ensuring the continuity of the paths of $\mathcal{Z}_a(\cdot)$ for any $a\geq 0$.

\vspace{0.2cm}

\noindent Combining our present Lemma \ref{ReducedProcesses} and Theorem 2.2 in \cite{AK24LocCrtiGen} with $\kappa>2$, one can deduce the following quenched convergence: for any continuous and bounded function $F:\R^3\mapsto\R$, in $\Pb$-probability
\begin{align}\label{QuenchedConv}
    \E^{\mathcal{E}}\Big[F\Big(\frac{1}{n}\VectCoord{R}{n}_{\an},\;\frac{1}{n}\VectCoord{Z}{n}_{\an},\;\frac{1}{n}\VectCoord{L}{n}_{\an}\Big)\Big]\underset{n\to\infty}{\longrightarrow}\E^{\mathcal{E}}\Big[F\big(\bm{c}_{\infty}W_{\infty}\mathcal{Y}_{a/W_{\infty}},\;W_{\infty}\mathcal{Y}_{a/W_{\infty}},\;2W_{\infty}\mathcal{Y}_{a/W_{\infty}};\; a\geq 0\big)\Big],
\end{align}
The convergence in \eqref{QuenchedConv} can be extended to the convergence of the finite-dimensional marginals in the sequence $(\frac{1}{n}\VectCoord{R}{n}_{\an},\;\frac{1}{n}\VectCoord{Z}{n}_{\an},\;\frac{1}{n}\VectCoord{L}{n}_{\an};\;  a\geq 0)$ but our argument does not provide a quenched version of Theorem \ref{ThMultiMar}.

\noindent 

\subsection{A genealogy problem for critical generations}\label{GenealogyProblem}

Let $b>0$ and $p\geq 1$. Introduce the random variable $(\VectCoord{\mathcal{X}}{1,p}_{\bn},\VectCoord{\mathcal{X}}{2,p}_{\bn})$ taking values in $\mathcal{R}^{(p)}_{\bn}\times\mathcal{R}^{(p)}_{\bn}$ with law defined by: for any $(x,y)\in\mathcal{U}\times\mathcal{U}$
\begin{align}\label{DefSampling}
    \P\Big(\big(\VectCoord{\mathcal{X}}{1,p}_{\bn},\VectCoord{\mathcal{X}}{2,p}_{\bn}\big)=(x,y)\Big):=\E\left[\frac{1}{\big(\VectCoord{R}{p}_{\bn}\big)^2}\un_{\big\{(x,y)\in\mathcal{R}^{(p)}_{\bn}\times\mathcal{R}^{(p)}_{\bn}\big\}}\Big|\VectCoord{R}{p}_{\bn}>0\right].
\end{align}
Under the Assumptions \ref{Assumption1} and \ref{Assumption2}, $\P(R^{(1)}_{\bn}>0)$ is positive so $(\VectCoord{\mathcal{X}}{1,p}_{\bn},\VectCoord{\mathcal{X}}{2,p}_{\bn})$ is well defined.

\vspace{0.1cm}

\noindent Note that $(\VectCoord{\mathcal{X}}{1,p}_{\bn},\VectCoord{\mathcal{X}}{2,p}_{\bn})$ is nothing but a couple of vertices picked uniformly in the set $\mathcal{R}^{(p)}_{\bn}\times\mathcal{R}^{(p)}_{\bn}$, conditionally on the event $\{\textrm{the generation }\bn\textrm{ of the range }\mathcal{R}^{(p)}\textrm{ is non empty}\}$. We are interested in the behavior of $|\VectCoord{\mathcal{X}}{1,p}_{\bn}\land\VectCoord{\mathcal{X}}{2,p}_{\bn}|$ as $n\to\infty$, where we recall that $\VectCoord{\mathcal{X}}{1,p}_{\bn}\land\VectCoord{\mathcal{X}}{2,p}_{\bn}$ denotes the most recent common ancestor of the vertices $\VectCoord{\mathcal{X}}{1,p}_{\bn}$ and $\VectCoord{\mathcal{X}}{2,p}_{\bn}$. It turns out that with a non-trivial probability, the coalescence happens in a generation proportional to $\bn$ or, with a non-trivial probability, the coalescence occurs close to the root of $\T$. A similar phenomenon happens for the genealogy of extremal particles of branching Brownian motion, see \cite{GenBBM}.

\begin{theo}[Two vertices picked in generation $\bn$ of $\mathcal{R}^{(n)}$]\label{GenealogyDiffCritic}
Assume that the Assumptions \ref{Assumption1} and \ref{Assumption2} hold. Let $a\in(0,b)$. We have
\begin{align}\label{RecentPastCoa}
    \lim_{n\to\infty}\P\big(|\VectCoord{\mathcal{X}}{1,n}_{\bn}\land\VectCoord{\mathcal{X}}{2,n}_{\bn}|\geq\an\big)=\E\left[\frac{\sum_{j=1}^{N_{a,b}}(\xi_j)^2}{(\sum_{j=1}^{N_{a,b}}\xi_j)^2}\Big| N_{a,b}>0\right],
\end{align}
with $N_{a,b}:=\sum_{i=1}^{N_b}G^i_{a,b}$, where $N_b$ is, under $\P^{\mathcal{E}}$, a Poisson random variable with parameter $W_{\infty}/(bc_0)$, $(G^i_{a,b})_{i\geq 1}$ is a sequence of\textrm{ i.i.d }Geometric random variables on $\N^*$ with probability of success $1-a/b$. $(\xi_j)_{j\geq 1}$ is a sequence of\textrm{ i.i.d }Exponential random variables with mean $1$. Besides, all random variables involved are independent. $N_{a,b}$ is also known as a Geometric-Poisson (or Pólya-Aeppli) random variable with parameter $(W_{\infty}/(bc_0),1-a/b)$. \\
Also, for any $m\in\N^*$
\begin{align}\label{RemotePastCoa}
    \lim_{n\to\infty}\P\big(|\VectCoord{\mathcal{X}}{1,n}_{\bn}\land\VectCoord{\mathcal{X}}{2,n}_{\bn}|<m\big) =1-\E\left[\frac{\sum_{|u|=m}\Big(Z^{(u)}_{\infty}\mathcal{Y}^{(u)}_{b/Z^{(u)}_{\infty}}\Big)^2}{\Big(\sum_{|u|=m}Z^{(u)}_{\infty}\mathcal{Y}^{(u)}_{b/Z^{(u)}_{\infty}}\Big)^2}\Bigg|\max_{|u|=m}\mathcal{Y}^{(u)}_{b/Z^{(u)}_{\infty}}>0\right],
\end{align}
where, for almost-every environment $\mathcal{E}$, $(\mathcal{Y}^{(z)};\; |z|=m)$ is a collection of\textrm{ i.i.d }copies under $\P^{\mathcal{E}}$ of $\mathcal{Y}$ (see \eqref{LaplaceCSBP}), $(W^{(z)}_{\infty};\; |z|=m)$ is a collection of\textrm{ i.i.d }copies of $W_{\infty}$ and independent of $((z,V(z));\; |z|=m)$ and $Z^{(u)}_{\infty}:=e^{-V(u)}W^{(u)}_{\infty}$. \\
\noindent Moreover
\begin{align}\label{SommeVaut1}
    \lim_{a\to 0}\;\lim_{n\to\infty}\P(|\VectCoord{\mathcal{X}}{1,n}_{\bn}\land\VectCoord{\mathcal{X}}{2,n}_{\bn}|\geq\an)+\lim_{m\to\infty}\;\lim_{n\to\infty}\P(|\VectCoord{\mathcal{X}}{1,n}_{\bn}\land\VectCoord{\mathcal{X}}{2,n}_{\bn}|<m)=1.
\end{align}
\end{theo}
\noindent Equation \eqref{RecentPastCoa} is strongly reminiscent of Theorem 2.1 in \cite{Athreya_SUB_CT} where the author dealt with the genealogical structure of a regular critical Galton-Watson tree. In this latter case, $N_{a,b}$ is a Geometric random variable whereas, in our case, $N_{a,b}$ is a sum of a Poisson number of Geometric random variables and in view of equation \eqref{RemotePastCoa}, this makes a huge difference. Indeed, the coalescence also occurs close to the root with a positive probability. Hence, we have a mixture between what we observe when we pick two vertices uniformly in any generation $l_n$ with $l_n\to\infty$ as $n\to\infty$ in a regular supercritical Galton-Watson tree (see Theorem 2 in \cite{Athreya_SUPER_CT} for instance) and what we observe when we pick two vertices uniformly in the generation $\bn$ of a regular critical Galton-Watson tree. 

\vspace{0.2cm}

\noindent Let us present alternative expressions of \eqref{RecentPastCoa} and \eqref{RemotePastCoa}. Note that if $\psi_b(\lambda):=\Eb[e^{-\lambda\frac{W_{\infty}}{bc_0}}]$ and $\psi^{(j)}_b(\lambda):=\frac{\mathrm{d}^j\psi_b(\lambda)}{\mathrm{d}\lambda^j}$, then we have a power series form of \eqref{RecentPastCoa}:
\begin{align}\label{ExpressionSerie}
    \E\left[\frac{\sum_{j=1}^{N_{a,b}}(\xi_j)^2}{(\sum_{j=1}^{N_{a,b}}\xi_j)^2}\Big| N_{a,b}>0\right]=\frac{2}{1-\psi_b(1)}\sum_{k\geq 1}\frac{1}{b^k(k+1)}\sum_{j=1}^k\binom{k-1}{j-1}\frac{(-1)^j\psi^{(j)}_b(1)}{j!}a^{k-j}(b-a)^j,
\end{align}
which is reminiscent but deeply different of formula (4) in \cite{DurrettGenealogy} for critical Galton-Watson trees. We also have the following integral expression for \eqref{RemotePastCoa}:
\begin{align}\label{RemotePastCoa2}
    &\E\left[\frac{\sum_{|u|=m}\Big(Z^{(u)}_{\infty}\mathcal{Y}^{(u)}_{b/Z^{(u)}_{\infty}}\Big)^2}{\Big(\sum_{|u|=m}Z^{(u)}_{\infty}\mathcal{Y}^{(u)}_{b/Z^{(u)}_{\infty}}\Big)^2}\Bigg|\max_{|u|=m}\mathcal{Y}^{(u)}_{b/Z^{(u)}_{\infty}}>0\right]\nonumber \\[0.4em] & =\frac{1}{1-\psi_b(1)}\int_0^{\infty}\frac{\lambda}{(1+\lambda bc_0)^4}\Eb\left[\sum_{|u|=m}\Big((Z^{(u)}_{\infty})^2+2bc_0(1+\lambda bc_0)Z^{(u)}_{\infty}\Big)e^{-\frac{\lambda}{1+\lambda bc_0}\sum_{|u|=m}Z^{(u)}_{\infty}}\right]\mathrm{d}\lambda.
\end{align}

\begin{remark}\label{RemSmallGenerations2}
One can see that
\begin{align}\label{RemSmallGenerations21}
    \lim_{b\to 0}\;\underset{a<b}{\lim_{a\to 0}}\;\lim_{n\to\infty}\P\big(|\VectCoord{\mathcal{X}}{1,n}_{\bn}\land\VectCoord{\mathcal{X}}{2,n}_{\bn}|\geq\an\big)=0,
\end{align}
and if $\mathcal{A}_{n,m}:=\lim_{n\to\infty}\sum_{x\not= y;|x|=|y|=n}e^{-V(x)}e^{-V(y)}\un_{\{|x\land y|<m\}}$, then
\begin{align}\label{RemSmallGenerations22}
    \lim_{b\to 0}\;\lim_{n\to\infty}\P\big(|\VectCoord{\mathcal{X}}{1,n}_{\bn}\land\VectCoord{\mathcal{X}}{2,n}_{\bn}|<m\big)=\Eb^*\Big[\frac{\mathcal{A}_{n,m}}{(W_{\infty})^2}\Big].
\end{align}
The right-hand side in \eqref{RemSmallGenerations22} corresponds to the probability that two vertices picked uniformly in a small generation for $\mathcal{R}^{(n)}$ (that is a generation $\ell_n$ such that $\ell_n=o(n)$ and $\ell_n\geq\delta_0^{-1}\log n$ for some $\delta_0>0$ ) have their most recent common ancestor below generation $m$, see Theorem 1.3 in \cite{AK23} for $k=2$ and $\kappa>4$. In particular, we observe some kind of continuity between small generations and critical generations, see section \ref{LastSection}.
\end{remark}

\begin{remark}[Two vertices picked in generation $\bn$ of $\mathcal{R}^{(1)}$]\label{RemarkGen1Excu0}
It is interesting to note that sampling two vertices uniformly in generation $\bn$ in $\mathcal{R}^{(1)}$ leads to a very different result. Assume that the Assumptions \ref{Assumption1} and \ref{Assumption2} hold and let $a\in(0,b)$. We have
\begin{align}\label{RemarkGen1Excu}
    \lim_{n\to\infty}\P\big(|\VectCoord{\mathcal{X}}{1,1}_{\bn}\land\VectCoord{\mathcal{X}}{2,1}_{\bn}|\geq\an\big)=\E\left[\frac{\sum_{j=1}^{G_{a,b}}(\xi_j)^2}{(\sum_{j=1}^{N_{a,b}}\xi_j)^2}\right]=\frac{2(b-a)}{a^2}\Big(b\log\Big(\frac{b}{b-a}\Big)-a\Big).
\end{align}
In particular, the coalescence can only occur in the recent past, see section \ref{LastSection}. Note that in the present case, the genealogy structure is the same as for a critical Galton-Watson trees, see Theorem 2.1 in \cite{Athreya_SUB_CT} and also \cite{HarrisJohnstonRoberts1}.
\end{remark}

\subsection{Range and genealogy in the sub-diffusive case}

When $\kappa\in(1,2]$, that is the sub-diffusive regime for the random walk $\X$, we prove in an upcoming work that the following convergence holds for the finite-dimensional marginals, under $\P$ and under $\P^{\mathcal{E}}$ for almost every environment $\mathcal{E}$:
\begin{align*}
    \Big(\frac{1}{n}\VectCoord{R}{n}_{\lfloor a\mathfrak{z}_n\rfloor},\;\frac{1}{n}\VectCoord{L}{n}_{\lfloor a\mathfrak{z}_n\rfloor};\;  a\geq 0\Big)\underset{n\to\infty}{\longrightarrow}\big(\bm{c}_{\infty}W_{\infty}\mathcal{Y}^{(\kappa)}_{a/W_{\infty}},\;2W_{\infty}\mathcal{Y}^{(\kappa)}_{a/W_{\infty}};\; a\geq 0\big),
\end{align*}
where, $\mathfrak{z}_n:=n^{\kappa\land 2-1}$ if $\kappa\not=2$ and $\mathfrak{z}_n:=n/\log n$ if $\kappa=2$. Also, $(\mathcal{Y}^{(\kappa)}_a;\; a\geq 0)$ is a CSBP with branching mechanism $\lambda\mapsto C_{\infty}\bm{\mathrm{C}}_{\kappa}\lambda^{\kappa\land 2}$ such that $\VectCoord{\mathcal{Y}_0}{\kappa}=1$ for some explicit constant $\bm{\mathrm{C}}_{\kappa}>0$, in the sense that for almost every environment $\mathcal{E}$, $(\mathcal{Y}^{(\kappa)}_a;\; a\geq 0)$ is a real-valued Markov process such that for any $\lambda\geq 0$ and any $0\leq a\leq b$
\begin{align*}
    \E^{\mathcal{E}}\Big[e^{-\lambda\mathcal{Y}^{(\kappa)}_b}\big|\mathcal{Y}^{(\kappa)}_a\Big]= \exp\Big(-\lambda\VectCoord{\mathcal{Y}}{\kappa}_a\big(1+(b-a)(\kappa\land 2-1)C_{\infty}\bm{\mathrm{C}}_{\kappa}\lambda^{\kappa\land 2-1}\big)^{-1/(\kappa\land 2-1)}\Big).
\end{align*} 
However, our arguments do not provide this convergence in law for the space of càdlàg functions $D([0,\infty),\R^2)$, even under the annealed probability $\P$. \\
We do expect the same genealogy structure as in the diffusive case (Theorem \ref{GenealogyDiffCritic} and Remark \ref{RemarkGen1Excu0}) but with different random variables involved. The genealogy in the sub-diffusive case will be treated in a future work.

\vspace{0.3cm}

\noindent The rest of the paper is organized as follows: in order to prove our results properly, we first present a few preliminary results, see section \ref{SectionPR}. After this, we show Theorem \ref{ThMultiMar}. We then prove Theorem \ref{GenealogyDiffCritic}. The end of the paper is dedicated to the continuity between critical generations and small generations, see Remark \ref{RemSmallGenerations2} and to vertices sampled in $\mathcal{R}^{(1)}$, see Remark \ref{RemarkGen1Excu0}.

\section{Proofs of the results}\label{SectionProofs}

\subsection{Preliminary results}\label{SectionPR}

Recall the definition of the random walk $(S_j)_{j\geq 0}$ in \eqref{LoiRW}. We have the following many-to-one lemma: for any $\ell\in\N^*$ and any measurable and non-negative function $\mathfrak{f}:\R^{\ell}\to\R$
\begin{align}\label{ManyTo1}
    \Eb\Big[\sum_{|x|=\ell}e^{-V(x)}\mathfrak{f}\big(V(x_1),\ldots,V(x_{\ell})\big)\Big]=\Eb\big[\mathfrak{f}\big(S_1,\ldots,S_{\ell}\big)\big].
\end{align}
Let $x\in\T$. It is well known that 
\begin{align}\label{ProbaAlpha}
    \alpha_x:=\P^{\mathcal{E}}(\VectCoord{N}{1}_x\geq 1)=\frac{e^{-V(x)}}{H_x},
\end{align}
where $H_x:=\sum_{e\leq w\leq x}e^{V(w)-V(x)}$. Moreover, under $\P^{\mathcal{E}}_{x^*}$, $\VectCoord{N}{1}_x$ follows a Geometric law on $\N$ with probability of success $\beta_x$ where 
\begin{align}\label{ProbaBeta}
    1-\beta_x:=\P^{\mathcal{E}}_{x^*}(\VectCoord{N}{1}_x\geq 1)=1-\frac{1}{H_x}.
\end{align}
In particular, $\E^{\mathcal{E}}[N^{(1)}_x]=e^{-V(x)}$.
\begin{lemm}\label{LemmJointLawEdge}
Let $x\not=y\in\T$ such that $|x|=|y|\geq 2$ and $|x\land y|=\ell$. We have
\begin{enumerate}[label=(\roman*)]
    \item\label{JointLaw1}
    \begin{align*}
     \E^{\mathcal{E}}\big[\big(\VectCoord{N}{1}_x\big)^2\big]=e^{-V(x)}(2H_x-1)\;\;\textrm{ and}\;\;\E^{\mathcal{E}}\big[\VectCoord{N}{1}_x\VectCoord{N}{1}_y\big]=2H_{x\land y}e^{V(x\land y)}e^{-V(x)}e^{-V(y)};
    \end{align*}
    \item\label{JointLaw2}
    \begin{align*}
    \E^{\mathcal{E}}\big[\VectCoord{N}{1}_x\un_{\{\VectCoord{N}{1}_y\geq 1\}}\big]=\Big(1+\frac{H_{y_{\ell+1},y}}{H_y}\Big)H_{x\land y}e^{V(x\land y)}e^{-V(x)}\frac{e^{-V(y)}}{H_y},
    \end{align*}
    where, for any $u\leq y$, $H_{u,y}:=\sum_{u\leq w\leq y}e^{V_u(w)-V_u(y)}$ and $V_u(y):=V(y)-V(u)$. 
\end{enumerate}
\end{lemm}

\begin{proof}
The proofs of \ref{JointLaw1} and \ref{JointLaw2} are similar so we only deal with the one of \ref{JointLaw2}.\\
First, let $u\not=v\in\T$ be two vertices having the same parent, that is 
$u^*=v^*=z$. Thanks to Fact \ref{FactGWMulti}, we have, for any $s,t\in[0,1]$
\begin{align*}
    \E^{\mathcal{E}}\big[s^{\VectCoord{N}{1}_u}t^{\VectCoord{N}{1}_v}\big]=\E^{\mathcal{E}}\big[a_{u,v}(s,t)^{\VectCoord{N}{1}_z}\big],
\end{align*}
where $a_{u,v}(s,t):=(1+(1-s)e^{-V_z(u)}+(1-t)e^{-V_z(v)})^{-1}$. Recalling that under $\P^{\mathcal{E}}_{z^*}$, $\VectCoord{N}{1}_z$ follows a Geometric law on $\N$ with probability of success $\beta_z$, the strong Markov property yields for any $t\in[0,1]$
\begin{align*}
    \E^{\mathcal{E}}\big[t^{\VectCoord{N}{1}_z}\big]=1-\frac{\alpha_z(1-t)}{1-t(1-\beta_z)},
\end{align*}
thus giving, together with \eqref{ProbaAlpha} and \eqref{ProbaBeta} 
\begin{align}\label{JointLawEdge}
    \E^{\mathcal{E}}\big[\VectCoord{N}{1}_ut^{\VectCoord{N}{1}_v}\big]=\frac{e^{-V_z(u)}\alpha_z\beta_z}{\big((1-t)e^{-V_z(v)}+\beta_z\big)^2}=\frac{e^{-V(u)}}{\big((1-t)H_ze^{-V_z(v)}+1\big)^2}.
\end{align}
Now, one can see that for any $x\not=y\in\T$ such that $|x|=|y|\geq 2$ and $|x\land y|=\ell$, we have
\begin{align*}
    \E^{\mathcal{E}}\big[\VectCoord{N}{1}_x\un_{\{\VectCoord{N}{1}_y\geq 1\}}\big]=\E^{\mathcal{E}}\Big[\VectCoord{N}{1}_x\Big]-\E^{\mathcal{E}}\Big[\VectCoord{N}{1}_x\un_{\{\VectCoord{N}{1}_y=0\}}\Big].
\end{align*}
If $\ell=k-1$, that is $x^*=y^*=z$, then, \eqref{JointLawEdge} with $t=0$ yields
\begin{align*}
     \E^{\mathcal{E}}\big[\VectCoord{N}{1}_x\un_{\{\VectCoord{N}{1}_y\geq 1\}}\big]=e^{-V(x)}\big(1-(H_ze^{-V_z(y)}+1)^{-2}\big),
\end{align*}
that is, using that $H_y=H_ze^{-V_z(y)}+1$
\begin{align*}
     \E^{\mathcal{E}}\big[\VectCoord{N}{1}_x\un_{\{\VectCoord{N}{1}_y\geq 1\}}\big]=e^{-V(x)}H_ze^{V(z)}\frac{e^{-V(y)}}{H_y}\frac{H_ze^{-V_z(y)}+2}{H_y}=H_ze^{V(z)}e^{-V(x)}\frac{e^{-V(y)}}{H_y}\Big(1+\frac{1}{H_y}\Big),
\end{align*}
which is exactly what we wanted. Otherwise, if $\ell<k-1$, then, thanks to Fact \ref{FactGWMulti}
\begin{align*}
    \E^{\mathcal{E}}\Big[\VectCoord{N}{1}_x\un_{\{\VectCoord{N}{1}_y=0\}}\Big]=e^{-V_{x_{k-1}}(x)}\E^{\mathcal{E}}\Big[\VectCoord{N}{1}_{x_{k-1}}\Tilde{a}_y(0)^{\VectCoord{N}{1}_{y_{k-1}}}\Big],
\end{align*}
with $\Tilde{a}_u(s):=(1+(1-s)e^{-V_{u^*}(u)})^{-1}$. By induction
\begin{align*}
    \E^{\mathcal{E}}\Big[\VectCoord{N}{1}_x\un_{\{\VectCoord{N}{1}_y=0\}}\Big]=e^{-V_{x_{\ell+1}}(x)}\E^{\mathcal{E}}\Big[\VectCoord{N}{1}_{x_{\ell+1}}\Tilde{a}_{y_{\ell+2},y_{k-1}}(\Tilde{a}_y(0))^{\VectCoord{N}{1}_{y_{\ell+1}}}\Big],
\end{align*}
where, for any $u\leq u$, $\Tilde{a}_{u,v}:=\Tilde{a}_u\circ\cdots\circ\Tilde{a}_v$. One can see that
\begin{align*}
    a_{y_{\ell+2},y_{k-1}}(s)=\frac{1+(1-s)\sum_{i=\ell+2}^{k-2}e^{-V_{y_i}(y_{k-1})}}{1+(1-s)\sum_{i=\ell+1}^{k-2}e^{-V_{y_i}(y_{k-1})}},
\end{align*}
thus giving
\begin{align*}
    a_{y_{\ell+2},y_{k-1}}(a_y(0))=\frac{H_{y_{\ell+2},y}}{H_{y_{\ell+1},y}}.
\end{align*}
Equation \eqref{JointLawEdge} with $t=H_{y_{\ell+2},y}/H_{y_{\ell+1},y}$, together with the fact that $H_y=H_ze^{-V_z(y)}+H_{y_{\ell+1},y}$, yield the result and the proof is completed.
\end{proof}

\noindent We now present a very simple maximal inequality that we will be using several times.
\begin{lemm}\label{InegMax}
Let $(K_j)_{j\geq 0}$ be a sequence of non-negative random variables under a probability measure $\Pbis$. For any $n\in\N$ and $\theta,a>0$
\begin{align*}
    \Pbis\Big(\max_{j\leq n}K_j\geq a\Big)\leq\frac{1}{a^{\theta}}\sup_{j\leq n}\Ebis\big[(K_j)^{\theta}\big].
\end{align*}
\end{lemm}
\begin{proof}
Let $\sigma_a:=\inf\{j\geq 0;\; K_j\geq a\}$. Since $K_j\geq 0$ for all $j\geq 0$, we have
\begin{align*}
    a^{\theta}\Pbis\Big(\max_{j\leq n}K_j\geq a\Big)\leq\Ebis\big[(K_{\sigma_a})^{\theta}\un_{\{\sigma_a\leq n\}}\big]\leq\sup_{j\leq n}\Ebis\big[(K_j)^{\theta}\big],
\end{align*}
which gives the result.
\end{proof}

\noindent Let us finally define a reduced version of the range $R^{(p)}_k$. Introduce 
\begin{align*}
    \Tilde{R}^{(p)}_k:=\sum_{j=1}^p\;\sum_{|x|=k}\un_{\{N_x^{(j)}-N_x^{(j-1)}\geq 1\}}.
\end{align*}
Let $(\gamma_n)_{n\geq 1}$ be a non-decreasing sequence of integers and define
\begin{align}\label{EnsSingleExcu}
    \mathfrak{S}_n:=\{\Tilde{R}^{(p)}_{k}=R^{(p)}_{k}\;\forall\; k\geq\gamma_n,\; \forall\; p\leq n\}.
\end{align}
\begin{lemm}\label{SingleExcu}
There exists a constant $C_{\eqref{SingleExcu}}>0$ such that for any $n\geq 1$
\begin{align*}
    \P\big(\mathfrak{S}_n\big)\geq 1-C_{\eqref{SingleExcu}}n^3e^{\gamma_n\psi(2)}.
\end{align*}
\end{lemm}
\begin{proof}
Indeed denote by $E^{(p)}_x:=\sum_{j=1}^p\un_{\{N^{(j)}_x-N_x^{(j-1)}\geq 1\}}$ the number of excursions above $e^*$ during which the vertex $x$ is visited. We have
\begin{align*}
    \P^{\mathcal{E}}\big(E_x^{(p)}\in\{0,1\}\;\forall\; |x|\geq\gamma_n, \forall\; p\leq n\big)\geq 1-\sum_{k\geq\gamma_n}\;\sum_{p\leq n}\;\sum_{|x|=k}\P^{\mathcal{E}}\big(E^{(p)}_x\geq 2\big),
\end{align*}
and thanks to the strong Markov property, $\{N_x^{(j)}-N_x^{(j-1)};\; 1\leq j\leq n\}$ is a collection of $n$\textrm{ i.i.d }copies of $N_x^{(1)}$ under $\P^{\mathcal{E}}$ so by \eqref{ProbaAlpha} 
\begin{align*}
    \P^{\mathcal{E}}\big(E^{(p)}_x\geq 2\big)=\P^{\mathcal{E}}\big(E^{(p)}_x\geq 1\big)-\P^{\mathcal{E}}\big(E^{(p)}_x=1\big)&\leq\E^{\mathcal{E}}\big[E^{(p)}_x\big]-p\P^{\mathcal{E}}\big(N^{(1)}_x\geq 1\big)\big(N^{(1)}_x=0\big)^{p-1} \\ & \leq pe^{-V(x)}\big(1-\big(1-e^{-V(x)}\big)^{p-1}\big)\leq p^2e^{-2V(x)}.
\end{align*}
Hence
\begin{align}\label{ProbaSingleExcu}
    \P\big(E_x^{(p)}\in\{0,1\}\;\forall\; |x|\geq\gamma_n, \forall\; p\leq n\big)\geq 1-n^3\sum_{k\geq\gamma_n}e^{k\psi(2)}\geq 1-n^3\frac{e^{\gamma_n\psi(2)}}{1-e^{\psi(2)}}.
\end{align}
Now, on the event $\{E_x^{(q)}\in\{0,1\}\;\forall\; |x|\geq\gamma_n, \forall\; q\leq n\}$, we have
\begin{align*}
    R^{(p)}_k=\sum_{j=1}^p\;\sum_{|x|=k}\un_{\{N^{(j)}_x-N^{(j-1)}_x\geq 1,\; \forall i\in\{1,\ldots,p\}\setminus\{j\}:\; N_x^{(i)}=N_x^{(i-1)}\}},
\end{align*}
so
\begin{align*}
    &\P^{\mathcal{E}}\big(\mathfrak{S}_n,\; E_x^{(q)}\in\{0,1\}\;\forall\; |x|\geq\gamma_n, \forall\; q\leq n\big) \\[0.7em] & \geq 1-\sum_{p\leq n}\;\sum_{k\geq\gamma_n}\;\sum_{|x|=k}p\P^{\mathcal{E}}\big(N^{(1)}_x\geq 1\big)\P^{\mathcal{E}}\big(N^{(1)}_x=0\big)^{p-1} \\[0.7em] & \geq 1-\sum_{p\leq n}\;\sum_{k\geq\gamma_n}\;\sum_{|x|=k}pe^{-V(x)}\Big(1-\big(1-e^{-V(x)}\big)^{p-1}\Big)\geq 1-n^3\sum_{k\geq\gamma_n}\;\sum_{|x|=k}e^{-2V(x)}. 
\end{align*}
Hence, by \eqref{ProbaSingleExcu}
\begin{align*}
    \P\big(\mathfrak{S}_n\big)\geq 1-2n^3\frac{e^{\gamma_n\psi(2)}}{1-e^{\psi(2)}},
\end{align*}
thus giving the result. 
\end{proof}

\subsection{Proof of Theorem \ref{ThMultiMar}}

\begin{lemm}\label{ReducedProcesses}
Assume that Assumption \ref{Assumption1} holds. For any $\varepsilon,M>0$
\begin{align*}
    \P\Big(\sup_{a\in[n^{-1/2},M]}\Big|R^{(n)}_{\an}-\bm{c}_{\infty}Z^{(n)}_{\an}\Big|>n\varepsilon\Big)\underset{n\to\infty}{\longrightarrow}0,
\end{align*}
and
\begin{align*}
    \P\Big(\sup_{a\in[n^{-1/2},M]}\Big|L^{(n)}_{\an}-2Z^{(n)}_{\an}\Big|>n\varepsilon\Big)\underset{n\to\infty}{\longrightarrow}0.
\end{align*}
\end{lemm}
\begin{proof}
We first first prove that
\begin{align}\label{ConvEspCarreRange2}
    \sup_{a\in[n^{-1/2},M]}\frac{1}{n^2}\E\Big[\Big|\Tilde{R}^{(n)}_{\an}-\bm{c}_{\infty}Z^{(n)}_{\an}\Big|^2\Big]\underset{n\to\infty}{\longrightarrow}0.
\end{align}
Recall that under $\P^{\mathcal{E}}$, $\Tilde{R}^{(n)}_{\an}$ is a sum of $n$\textrm{ i.i.d }copies of $R^{(1)}_{\an}$. Hence
\begin{align*}
    \E\Big[\Big|\Tilde{R}^{(n)}_{\an}-\bm{c}_{\infty}Z^{(n)}_{\an}\Big|^2\Big]=&\sum_{i\not=j=1}^n\Big(\Eb\Big[\E^{\mathcal{E}}\big[R^{(1)}_{ \an}\big]^2\Big]+\bm{c}_{\infty}^2\Eb\Big[\E^{\mathcal{E}}\big[Z^{(1)}_{ \an}\big]^2\Big]-2\bm{c}_{\infty}\Eb\Big[\E^{\mathcal{E}}\big[R^{(1)}_{ \an}\big]\E^{\mathcal{E}}\big[Z^{(1)}_{\an}\big]\Big]\Big) \\ & +\sum_{i=1}^n\Big(\E\big[\big(R^{(1)}_{ \an}\big)^2\big]+\bm{c}_{\infty}^2\E\big[\big(Z^{(1)}_{ \an}\big)^2\big]-2\bm{c}_{\infty}\E\big[R^{(1)}_{ \an}Z^{(1)}_{ \an}\big]\Big).
\end{align*}
We first deal with $\Eb[\E^{\mathcal{E}}[R^{(1)}_{ \an}]^2]+\bm{c}_{\infty}^2\Eb[\E^{\mathcal{E}}[Z^{(1)}_{ \an}]^2]-2\bm{c}_{\infty}\Eb[\E^{\mathcal{E}}[R^{(1)}_{ \an}]\E^{\mathcal{E}}\big[Z^{(1)}_{\an}]]$. By \eqref{ProbaAlpha}, we have
\begin{align*}
    \Eb\Big[\E^{\mathcal{E}}\big[R^{(1)}_{ \an}\big]^2\Big]=\Eb\Big[\Big(\sum_{|x|= \an}\frac{e^{-V(x)}}{H_x}\Big)^2\Big]=\Eb\Big[\sum_{|x|=|y|=\an}\frac{e^{-V(x)}}{H_x}\frac{e^{-V(y)}}{H_y}\Big],
\end{align*}
and since $H_x\geq 1$
\begin{align*}
    \Eb\Big[\sum_{|x|=|y|=\an}\frac{e^{-V(x)}}{H_x}\frac{e^{-V(y)}}{H_y}\Big]&=\Eb\Big[\sum_{|x|=\an}\frac{e^{-2V(x)}}{(H_x)^2}\Big]+\Eb\Big[\underset{|x|=|y|=\an}{\sum_{x\not=y}}\frac{e^{-V(x)}}{H_x}\frac{e^{-V(y)}}{H_y}\Big] \\ & \leq e^{\an\psi(2)}+\Eb\Bigg[\underset{|x|=|y|=\an}{\sum_{x\not=y}}\frac{e^{-V(x)}}{H_x}\frac{e^{-V(y)}}{H_y}\Bigg].
\end{align*}
For any $x\in\T$ and any $u\leq x$, one can notice that $H_x=(H_u-1)e^{-V_u(x)}+H_{u,x}$ where we recall that $V_u(x)=V(x)-V(u)$ and $H_{u,x}=\sum_{u\leq w\leq x}e^{V_u(w)-V_u(x)}$. Hence
\begin{align*}
    &\Eb\Bigg[\underset{|x|=|y|= \an}{\sum_{x\not=y}}\frac{e^{-V(x)}}{H_x}\frac{e^{-V(y)}}{H_y}\Bigg] \\ &\leq\sum_{k< \an}\Eb\Bigg[\sum_{|z|=k}e^{-2V(z)}\underset{u^*=v^*=z}{\sum_{u\not=v}}e^{-V_z(u)}e^{-V_z(v)}\underset{x\geq u}{\sum_{|x|= \an}}\frac{e^{-V_u(x)}}{H_{u,x}}\underset{y\geq v}{\sum_{|y|= \an}}\frac{e^{-V_v(y)}}{H_{v,y}}\Bigg].   
\end{align*}
By independence of the increments of the branching random walk $(\T,(V(x);\; x\in\T))$, we have
\begin{align*}
    &\Eb\Bigg[\sum_{|z|=k}e^{-2V(z)}\underset{u^*=v^*=z}{\sum_{u\not=v}}e^{-V_z(u)}e^{-V_z(v)}\underset{x\geq u}{\sum_{|x|= \an}}\frac{e^{-V_u(x)}}{H_{u,x}}\underset{y\geq v}{\sum_{|y|= \an}}\frac{e^{-V_v(y)}}{H_{v,y}}\Bigg] \\ & =\Eb\Big[\sum_{|z|=k}e^{-2V(z)}\underset{u^*=v^*=z}{\sum_{u\not=v}}e^{-V_z(u)}e^{-V_z(v)}\Big]\Eb\Big[\sum_{|x|= \an-k-1}\frac{e^{-V(x)}}{H_x}\Big]^2 \\ & =e^{k\psi(2)}c_0\big(1-e^{\psi(2)}\big)\Eb\Big[\sum_{|x|= \an-k-1}\frac{e^{-V(x)}}{H_x}\Big]^2=e^{k\psi(2)}c_0\big(1-e^{\psi(2)}\big)\Eb\Big[\Big(\sum_{j\leq\an-k-1}e^{-S_j}\Big)^{-1}\Big],
\end{align*}
where we have used the many-to-one lemma \eqref{ManyTo1} for the last equality, and then
\begin{align*}
    \Eb\Bigg[\underset{|x|=|y|= \an}{\sum_{x\not=y}}\frac{e^{-V(x)}}{H_x}\frac{e^{-V(y)}}{H_y}\Bigg]\leq c_0\big(1-e^{\psi(2)}\big)\sum_{k<\an}e^{k\psi(2)}\Eb\Big[\Big(\sum_{j\leq\an-k-1}e^{-S_j}\Big)^{-1}\Big]^2.
\end{align*}
In view of this last equality, terms close enough to $ \an$ do not give a significant contribution to the sum, so we split it as follows
\begin{align*}
    \sum_{k<\an}e^{k\psi(2)}\Eb\Big[\Big(\sum_{j\leq\an-k-1}e^{-S_j}\Big)^{-1}\Big]^2&=\sum_{k< \an-\PEsqrt}e^{k\psi(2)}\Eb\Big[\Big(\sum_{j\leq\an-k-1}e^{-S_j}\Big)^{-1}\Big]^2 \\[0.7em] & +\sum_{k= \an-\PEsqrt}^{ \an-1}e^{k\psi(2)}\Eb\Big[\Big(\sum_{j\leq\an-k-1}e^{-S_j}\Big)^{-1}\Big]^2,
\end{align*}
which, using that the sequence of real numbers $(\Eb[(\sum_{j\leq\ell}e^{-S_j})^{-1}])_{\ell\in\N}$ is non-increasing and bounded by $1$, is smaller than
\begin{align*}
    \Eb\Big[\Big(\sum_{j\leq\PEsqrt}e^{-S_j}\Big)^{-1}\Big]^2\sum_{k< \an-\PEsqrt}e^{k\psi(2)}+\sqrtBis{n},
\end{align*}
and finally
\begin{align}\label{Partie1}
    \Eb\Big[\sum_{|x|=|y|=\an}\frac{e^{-V(x)}}{H_x}\frac{e^{-V(y)}}{H_y}\Big]\leq &c_0\big(1-e^{\psi(2)}\big)\Eb\Big[\Big(\sum_{j\leq\PEsqrt}e^{-S_j}\Big)^{-1}\Big]^2\sum_{k< \an-\PEsqrt}e^{k\psi(2)}\nonumber \\[0.7em] & +e^{\an\psi(2)}+c_0\big(1-e^{\psi(2)}\big)\sqrtBis{n}.
\end{align}
Similarly, we obtain
\begin{align}\label{Partie2}
    \Eb\Big[\E^{\mathcal{E}}\big[Z^{(1)}_{\an}\big]^2\Big]=\Eb\Big[\sum_{|x|=|y|=\an}e^{-V(x)}e^{-V(y)}\Big]\leq& c_0\big(1-e^{\psi(2)}\big)\sum_{k< \an-\PEsqrt}e^{k\psi(2)}+e^{\an\psi(2)}\nonumber \\[0.7em] & +c_0\big(1-e^{\psi(2)}\big)\sqrtBis{n}.
\end{align}
We also have
\begin{align*}
    \Eb\Big[\E^{\mathcal{E}}\big[R^{(1)}_{ \an}\big]\E^{\mathcal{E}}\big[Z^{(1)}_{\an}\big]\Big]=\Eb\Big[\sum_{|x|=|y|=\an}\frac{e^{-V(x)}}{H_x}e^{-V(y)}\Big]\geq\Eb\Bigg[\underset{|x|=|y|= \an}{\sum_{x\not=y}}\frac{e^{-V(x)}}{H_x}e^{-V(y)}\Bigg],
\end{align*}
and we need a lower bound for the latter mean. Again, noting that for any $x\in\T$ and any $u\leq x$, we have $H_x=(H_u-1)e^{-V_u(x)}+H_{u,x}$, we obtain, by independence of the increments of the branching random walk $(\T,(V(x);\; x\in\T))$ and $\psi(1)=0$, that
\begin{align*}
    \Eb\Bigg[\underset{|x|=|y|= \an}{\sum_{x\not=y}}\frac{e^{-V(x)}}{H_x}e^{-V(y)}\Bigg]&=\sum_{k< \an}\Eb\Bigg[\sum_{|z|=k}e^{-2V(z)}\underset{u^*=v^*=z}{\sum_{u\not=v}}e^{-V_z(u)}e^{-V_z(v)}\underset{x\geq u}{\sum_{|x|= \an}}\frac{e^{-V_u(x)}}{H_{x}}\underset{y\geq v}{\sum_{|y|= \an}}e^{-V_v(y)}\Bigg] \\ & \geq\sum_{k< \an-\PEsqrt}\Eb\Big[\sum_{|z|=k}e^{-2V(z)}\underset{u^*=v^*=z}{\sum_{u\not=v}}e^{-V_z(u)}e^{-V_z(v)}\varphi_{k,n}(H_v)\Big],
\end{align*}
where for any $h\geq 1$
\begin{align}\label{Defphi}
    \varphi_{k,n}(h):=\un_{\{h\leq n^{1/3}+1\}}\Eb\Big[\sum_{|x|= \an-k-1}\un_{\{V(x)\geq n^{1/3}\}}\frac{e^{-V(x)}}{(h-1)e^{-V(x)}+H_x}\Big].
\end{align}
Note that for any $h\leq n^{1/3}+1$ and $k<\an-\PEsqrt$, thanks to the many-to-one Lemma \eqref{ManyTo1}
\begin{align*}
    \Eb\Big[\sum_{|x|= \an-k-1}\un_{\{V(x)\geq n^{1/3}\}}\frac{e^{-V(x)}}{(h-1)e^{-V(x)}+H_x}\Big]\geq&\Eb\Big[\Big(n^{1/3}e^{-n^{1/3}}+\sum_{j\leq \an}e^{-S_j}\Big)^{-1}\Big] \\ & -\Pb\big(S_{ \an-k-1}<n^{1/3}\big).
\end{align*}
By Markov inequality
\begin{align*}
    \Pb\big(S_{ \an-k-1}<n^{1/3}\big)=\Pb\big(e^{-S_{ \an-k-1}}>e^{-n^{1/3}}\big)\leq e^{\PEsqrt\psi(2)+n^{1/3}},
\end{align*}
where we have used recall that $\psi(2)<0$ since we are in the diffusive case. Therefore
\begin{align*}
    \Eb\Bigg[\underset{|x|=|y|= \an}{\sum_{x\not=y}}\frac{e^{-V(x)}}{H_x}e^{-V(y)}\Bigg]\geq\Big(&\Eb\Big[\Big(n^{1/3}e^{-n^{1/3}}+\sum_{j\leq \an}e^{-S_j}\Big)^{-1}\Big]-e^{\PEsqrt\psi(2)+n^{1/3}}\Big) \\ & \times\sum_{k< \an-\PEsqrt}\Eb\Big[\sum_{|z|=k}e^{-2V(z)}\underset{u^*=v^*=z}{\sum_{u\not=v}}e^{-V_z(u)}e^{-V_z(v)}\un_{\{H_v\leq n^{1/3}+1\}}\Big].
\end{align*}
Decompose $H_u=H_ze^{-V_z(u)}+1$ yields that $\Eb[\sum_{|z|=k}e^{-2V(z)}\sum_{u\not=v;u^*=v^*=z}e^{-V_z(u)}e^{-V_z(v)}\un_{\{H_u\leq n^{1/3}+1\}}]$ is larger than
\begin{align}\label{TwoExpectations}
    &\Eb\Big[\sum_{|z|=k}e^{-2V(z)}\underset{u^*=v^*=z}{\sum_{u\not=v}}e^{-V_z(u)}e^{-V_z(v)}\un_{\{H_z\leq n^{1/8},\;V_z(u)\geq-\frac{1}{8}\log n\}}\Big]\nonumber \\ & = \Eb\Big[\sum_{|z|=k}e^{-2V(z)}\un_{\{H_z\leq n^{1/8}\}}\Big]\Eb\Big[\underset{|u|=|v|=1}{\sum_{u\not=v}}e^{-V(u)}e^{-V(v)}\un_{\{V(u)\geq-\frac{1}{8}\log n\}}\Big],
\end{align}
where we have also used independence of the increments of the branching random walk $(\T,(V(x);\; x\in\T))$. For the first mean in \eqref{TwoExpectations}, note that thanks to the many-to-one Lemma \eqref{ManyTo1}, we have
\begin{align*}
    \Eb\Big[\sum_{|z|=k}e^{-2V(z)}\un_{\{H_z\leq n^{1/8}\}}\Big]&=e^{k\psi(2)}-\Eb\Big[e^{-S_k}\un_{\{\sum_{j\leq k}e^{S_j-S_k}>n^{1/8}\}}\Big] \\ & \geq e^{k\psi(2)}-e^{\frac{k}{1+\delta_3}\psi(2+\delta_3)}\Pb\Big(\sum_{j\leq k}e^{-S_j}>n^{1/8}\Big)^{\frac{\delta_3}{1+\delta_3}} \\ & \geq e^{k\psi(2)}-e^{\frac{k}{1+\delta_3}\psi(2+\delta_3)}\big(n^{-1/8}(1-e^{\psi(2)})^{-1}\big)^{\frac{\delta_3}{1+\delta_3}},
\end{align*}
where $\delta_3$ belongs to $(0,\delta_1)\cap\{t>0;\; \psi(2+t)<0\}$ which is nonempty since $\kappa>2$, the second inequality comes from the Hölder inequality and  the last one is due to Markov inequality.
For the second mean in \eqref{TwoExpectations}, we have
\begin{align*}
    \Eb\Big[\underset{|u|=|v|=1}{\sum_{u\not=v}}e^{-V(u)}e^{-V(v)}\un_{\{V(u)\geq-\frac{1}{8}\log n\}}\Big]=c_0\big(1-e^{\psi(2)}\big)-\Eb\Big[\underset{|u|=|v|=1}{\sum_{u\not=v}}e^{-V(u)}e^{-V(v)}\un_{\{V(u)<-\frac{1}{8}\log n\}}\Big],
\end{align*}
and
\begin{align*}
    \Eb\Big[\underset{|u|=|v|=1}{\sum_{u\not=v}}e^{-V(u)}e^{-V(v)}\un_{\{V(u)<-\frac{1}{8}\log n\}}\Big]\leq\Eb\Big[\Big(\sum_{|u|=1}e^{-V(u)}\Big)^2\un_{\{\max_{|w|=1}V(w)<-\frac{1}{8}\log n\}}\Big].
\end{align*}
Note that $\Pb(\max_{|w|=1}V(w)>-\frac{1}{8}\log n)\leq n^{-1/8}$ so, by assumption \ref{Assumption1}, the above mean goes to $0$ as $n$ goes to infinity. Hence, for $n$ large enough and any $a\in[n^{-1/2},M]$
\begin{align*}
    \Eb\Bigg[\underset{|x|=|y|= \an}{\sum_{x\not=y}}\frac{e^{-V(x)}}{H_x}e^{-V(y)}\Bigg]\geq\big(1+o(1)\big)&\Big(\Eb\Big[\Big(n^{1/3}e^{-n^{1/3}}+\sum_{j\leq \an}e^{-S_j}\Big)^{-1}\Big]-e^{\PEsqrt\psi(2)+n^{1/3}}\Big) \\ & \times c_0\big(1-e^{\psi(2)}\big)\sum_{k< \an-\PEsqrt}e^{k\psi(2)}.
\end{align*}
Combining this with \eqref{Partie1} and \eqref{Partie2}, we obtain
\begin{align}\label{ConvUnifPartie1}
    \sup_{a\in[n^{-1/2}\Mn]}\frac{1}{n^2}\sum_{i\not=j=1}^n\Big(\Eb\Big[\E^{\mathcal{E}}\big[R^{(1)}_{ \an}\big]^2\Big]+\bm{c}_{\infty}^2\Eb\Big[\E^{\mathcal{E}}\big[Z^{(1)}_{ \an}\big]^2\Big]-2\bm{c}_{\infty}\Eb\Big[\E^{\mathcal{E}}\big[R^{(1)}_{ \an}\big]\E^{\mathcal{E}}\big[Z^{(1)}_{\an}\big]\Big]\Big)\underset{n\to\infty}{\longrightarrow}0.
\end{align}
We now deal with $\E[(R^{(1)}_{ \an})^2]+\bm{c}_{\infty}^2\E[(Z^{(1)}_{ \an})^2]-2\bm{c}_{\infty}\E[R^{(1)}_{ \an}Z^{(1)}_{ \an}]$. Note that
\begin{align*}
    \E\big[\big(\VectCoord{R}{1}_{ \an}\big)^2\big]=&\Eb\Big[\sum_{|x|= \an}\P^{\mathcal{E}}\big(\VectCoord{N}{1}_x\geq 1\big)\Big]+\Eb\Bigg[\underset{|x|=|y|= \an}{\sum_{x\not=y}}\P^{\mathcal{E}}\big(\VectCoord{N}{1}_x\geq 1,\;\VectCoord{N}{1}_y\geq 1 \big)\Bigg] \\ & \leq\Eb\Big[\sum_{|x|= \an}\P^{\mathcal{E}}\big(\VectCoord{N}{1}_x\geq 1\big)\Big]+\Eb\Bigg[\underset{|x|=|y|= \an}{\sum_{x\not=y}}\E^{\mathcal{E}}\big[\VectCoord{N}{1}_x\VectCoord{N}{1}_y\big]\Bigg],
\end{align*}
so by \eqref{ProbaAlpha} and Lemma \ref{LemmJointLawEdge}, we obtain
\begin{align*}
    \E\big[\big(\VectCoord{R}{1}_{ \an}\big)^2\big]\leq 1+2\Eb\Bigg[\underset{|x|=|y|= \an}{\sum_{x\not=y}}H_{x\land y}e^{V(x\land y)}\frac{e^{-V(x)}}{H_x}\frac{e^{-V(y)}}{H_y}\Bigg].
\end{align*}
Again, for any $x\in\T$ and any $u\leq x$, one can notice that $H_x=(H_u-1)e^{-V_u(x)}+H_{u,x}$ so
\begin{align*}
    &\Eb\Bigg[\underset{|x|=|y|= \an}{\sum_{x\not=y}}H_{x\land y}e^{V(x\land y)}\frac{e^{-V(x)}}{H_x}\frac{e^{-V(y)}}{H_y}\Bigg] \\ &\leq\sum_{k< \an}\Eb\Bigg[\sum_{|z|=k}H_ze^{-V(z)}\underset{u^*=v^*=z}{\sum_{u\not=v}}e^{-V_z(u)}e^{-V_z(v)}\underset{x\geq u}{\sum_{|x|= \an}}\frac{e^{-V_u(x)}}{H_{u,x}}\underset{y\geq v}{\sum_{|y|= \an}}\frac{e^{-V_v(y)}}{H_{v,y}}\Bigg].   
\end{align*}
By independence of the increments of the branching random walk $(\T,(V(x);\; x\in\T))$, we have
\begin{align*}
    &\Eb\Bigg[\sum_{|z|=k}H_ze^{-V(z)}\underset{u^*=v^*=z}{\sum_{u\not=v}}e^{-V_z(u)}e^{-V_z(v)}\underset{x\geq u}{\sum_{|x|= \an}}\frac{e^{-V_u(x)}}{H_{u,x}}\underset{y\geq v}{\sum_{|y|= \an}}\frac{e^{-V_v(y)}}{H_{v,y}}\Bigg] \\ & =\Eb\Big[\sum_{|z|=k}H_ze^{-V(z)}\underset{u^*=v^*=z}{\sum_{u\not=v}}e^{-V_z(u)}e^{-V_z(v)}\Big]\Eb\Big[\sum_{|x|= \an-k-1}\frac{e^{-V(x)}}{H_x}\Big]^2 \\ & =\Eb\Big[\sum_{|z|=k}H_ze^{-V(z)}\Big]c_0\big(1-e^{\psi(2)}\big)\Eb\Big[\sum_{|x|= \an-k-1}\frac{e^{-V(x)}}{H_x}\Big]^2.
\end{align*}
Hence, the many-to-one Lemma \eqref{ManyTo1} yields
\begin{align*}
    \Eb\Bigg[\underset{|x|=|y|= \an}{\sum_{x\not=y}}H_{x\land y}e^{V(x\land y)}\frac{e^{-V(x)}}{H_x}\frac{e^{-V(y)}}{H_y}\Bigg]\leq c_0\sum_{k< \an}\big(1-e^{(k+1)\psi(2)}\big)\Eb\Big[\Big(\sum_{j\leq \an-k-1}e^{-S_j}\Big)^{-1}\Big]^2.
\end{align*}
In view of this last equality, terms close enough to $ \an$ do not give a significant contribution to the sum, so we split it as follows
\begin{align*}
    \sum_{k< \an}\big(1-e^{(k+1)\psi(2)}\big)\Eb\Big[\Big(\sum_{j\leq \an-k-1}e^{-S_j}\Big)^{-1}\Big]^2&=\sum_{k< \an-\PEsqrt}\big(1-e^{(k+1)\psi(2)}\big)\Eb\Big[\Big(\sum_{j\leq \an-k-1}e^{-S_j}\Big)^{-1}\Big]^2 \\[0.7em] & +\sum_{k= \an-\PEsqrt}^{ \an-1}\big(1-e^{(k+1)\psi(2)}\big)\Eb\Big[\Big(\sum_{j\leq \an-k-1}e^{-S_j}\Big)^{-1}\Big]^2,
\end{align*}
which, using that the sequence of real numbers $(\Eb[(\sum_{j\leq\ell}e^{-S_j})^{-1}])_{\ell\in\N}$ is non-increasing and bounded by $1$, is smaller than
\begin{align*}
    \Eb\Big[\Big(\sum_{j\leq\PEsqrt}e^{-S_j}\Big)^{-1}\Big]^2\sum_{k< \an-\PEsqrt}\big(1-e^{(k+1)\psi(2)}\big)+\sqrtBis{n}\leq an\Eb\Big[\Big(\sum_{j\leq\PEsqrt}e^{-S_j}\Big)^{-1}\Big]^2+\sqrtBis{n}.
\end{align*}
Hence
\begin{align*}
    \E\big[\big(\VectCoord{R}{1}_{ \an}\big)^2\big]\leq 1+2\sqrtBis{n}+2an\Eb\Big[\Big(\sum_{j\leq\PEsqrt}e^{-S_j}\Big)^{-1}\Big]^2.
\end{align*}
Note that
\begin{align*}
    \E\big[\big(\VectCoord{Z}{1}_{\an}\big)^2\big]=&\Eb\Bigg[\sum_{|x|= \an}\E^{\mathcal{E}}\big[\VectCoord{N}{1}_x\big]\Bigg]+\Eb\Big[\underset{|x|=|y|= \an}{\sum_{x\not=y}}\E^{\mathcal{E}}\big[\VectCoord{N}{1}_x\VectCoord{N}{1}_y\big]\Big] \\ & =1+2\Eb\Bigg[\underset{|x|=|y|= \an}{\sum_{x\not=y}}H_{x\land y}e^{V(x\land y)}e^{-V(x)}e^{-V(y)}\Bigg],
\end{align*}
thus giving
\begin{align*}
    \E\big[\big(\VectCoord{Z}{1}_{ \an}\big)^2\big]\leq 1+2\sqrtBis{n}+2an.
\end{align*}
To complete the proof, we need a lower bound for $\E[\VectCoord{R}{1}_{ \an}\VectCoord{Z}{1}_{ \an}]$. For that, one can see that
\begin{align*}
    \E[\VectCoord{R}{1}_{ \an}\VectCoord{Z}{1}_{ \an}]=\E[\VectCoord{R}{1}_{ \an}\VectCoord{Z}{1}_{ \an}]\geq\Eb\Bigg[\underset{|x|=|y|= \an}{\sum_{x\not=y}}\E^{\mathcal{E}}\Big[\VectCoord{N}{1}_x\un_{\{\VectCoord{N}{1}_x\geq 1\}}\Big]\Bigg],
\end{align*}
which, thanks to Lemma \ref{LemmJointLawEdge} \ref{JointLaw2} and together with the fact that $1\leq H_y=(H_v-1)e^{-V_v(y)}+H_{v,y}$, is larger than
\begin{align*}
    &\sum_{k< \an-\PEsqrt}\Eb\Big[\sum_{|z|=k}H_ze^{-V(z)}\underset{u^*=v^*=z}{\sum_{u\not=v}}e^{-V_z(u)}e^{-V_z(v)}\underset{x\geq u}{\sum_{|x|= \an}}e^{-V_u(x)}\underset{y\geq v}{\sum_{|y|= \an}}\frac{e^{-V_v(y)}}{H_{y}}\Big(2-e^{-V_v(y)}(H_v-1)\Big)\Big] \\ & \geq\sum_{k< \an-\PEsqrt}\Eb\Big[\sum_{|z|=k}H_ze^{-V(z)}\underset{u^*=v^*=z}{\sum_{u\not=v}}e^{-V_z(u)}e^{-V_z(v)}\Tilde{\varphi}_{k,n}(H_v)\Big],
\end{align*}
where we have also used independence of the increments of the branching random walk $(\T,(V(x);\; x\in\T))$ and that $\psi(1)=0$. For any $h\geq 1$
\begin{align*}
    \Tilde{\varphi}_{k,n}(h):=\un_{\{h\leq n^{1/3}+1\}}\Eb\Big[\sum_{|x|= \an-k-1}\un_{\{V(x)\geq n^{1/3}\}}\frac{e^{-V(x)}}{(h-1)e^{-V(x)}+H_x}\big(2-e^{-V(x)}(h-1)\big)\Big].
\end{align*}
Note that $\Tilde{\varphi}_{k,n}(h)\geq(2-n^{-1/3}e^{-n^{1/3}})\varphi_{k,n}(h)$, see \eqref{Defphi}. Hence
\begin{align*}
    \E\Big[\VectCoord{R}{1}_{ \an}\VectCoord{Z}{1}_{ \an}\Big]\geq&\big(2-n^{-1/3}e^{-n^{1/3}}\big)\Big(\Eb\Big[\Big(n^{1/3}e^{-n^{1/3}}+\sum_{j\leq \an}e^{-S_j}\Big)^{-1}\Big]-e^{\PEsqrt\psi(2)+n^{1/3}}\Big) \\ & \times\sum_{k< \an-\PEsqrt}\Eb\Big[\sum_{|z|=k}H_ze^{-V(z)}\underset{u^*=v^*=z}{\sum_{u\not=v}}e^{-V_z(u)}e^{-V_z(v)}\un_{\{H_v\leq n^{1/3}+1\}}\Big].
\end{align*}
As in equation \eqref{TwoExpectations}, the latter mean is larger than
\begin{align*}
    &\Eb\Big[\sum_{|z|=k}H_ze^{-V(z)}\underset{u^*=v^*=z}{\sum_{u\not=v}}e^{-V_z(u)}e^{-V_z(v)}\un_{\{H_z\leq n^{1/8},\;V_z(u)\geq-\frac{1}{8}\log n\}}\Big] \\ & = \Eb\Big[\sum_{|z|=k}H_ze^{-V(z)}\un_{\{H_z\leq n^{1/8}\}}\Big]\Eb\Big[\underset{|u|=|v|=1}{\sum_{u\not=v}}e^{-V(u)}e^{-V(v)}\un_{\{V(u)\geq-\frac{1}{8}\log n\}}\Big],
\end{align*}
For the first mean above, using that $H_z\geq 1$ and thanks to the many-to-one Lemma \eqref{ManyTo1}, we have
\begin{align*}
    \Eb\Big[\sum_{|z|=k}H_ze^{-V(z)}\un_{\{H_z\leq n^{1/8}\}}\Big]\geq\Eb\Big[\sum_{|z|=k}e^{-V(z)}\un_{\{H_z\leq n^{1/8}\}}\Big]&\geq 1-\Pb\Big(\sum_{j\leq k}e^{-S_j}>n^{1/8}\Big) \\ & \geq 1-n^{-1/8}(1-e^{\psi(2)})^{-1},
\end{align*}
where, as previously, the last inequality comes from Markov inequality. Hence, for $n$ large enough and any $a\in[n^{-1/2},M]$
\begin{align*}
    \sum_{k< \an-\PEsqrt}\Eb\Big[\sum_{|z|=k}H_ze^{-V(z)}\underset{u^*=v^*=z}{\sum_{u\not=v}}e^{-V_z(u)}e^{-V_z(v)}\un_{\{H_v\leq n^{1/3}+1\}}\Big]=an(1+o(1)),
\end{align*}
Collecting previous computations, $\E[(R^{(1)}_{ \an})^2]+\bm{c}_{\infty}^2\E[(Z^{(1)}_{ \an})^2]-2\bm{c}_{\infty}\E[R^{(1)}_{ \an}Z^{(1)}_{ \an}]$ is, for $n$ large enough, any $a\in[n^{-1/2},M]$ and any $i\geq 1$, smaller than
\begin{align*}
    &o(n)+2an\Big(\bm{c}_{\infty}^2+\Eb\Big[\Big(\sum_{j\leq\PEsqrt}e^{-S_j}\Big)^{-1}\Big]^2-2\bm{c}_{\infty}\Eb\Big[\Big(n^{1/3}e^{-n^{1/3}}+\sum_{j\leq \an}e^{-S_j}\Big)^{-1}\Big]\Big) \\ & \leq o(n)+2Mn\Big(\bm{c}_{\infty}^2+\Eb\Big[\Big(\sum_{j\leq\PEsqrt}e^{-S_j}\Big)^{-1}\Big]^2-2\bm{c}_{\infty}\Eb\Big[\Big(n^{1/3}e^{-n^{1/3}}+\sum_{j\leq\Mn-\lfloor(\log n)^2\rfloor}e^{-S_j}\Big)^{-1}\Big]\Big).
\end{align*}
Therefore
\begin{align*}
    \sup_{a\in[n^{-1/2},M]}\frac{1}{n^2}\sum_{i=1}^n\Big(\E\big[\big(R^{(1)}_{ \an}\big)^2\big]+\bm{c}_{\infty}^2\E\big[\big(Z^{(1)}_{ \an}\big)^2\big]-2\bm{c}_{\infty}\E\big[R^{(1)}_{ \an}Z^{(1)}_{ \an}\big]\Big)\underset{n\to\infty}{\longrightarrow}0.
\end{align*}
Combining this convergence with \eqref{ConvUnifPartie1}, we obtain \eqref{ConvEspCarreRange2}. It follows from Lemma \ref{InegMax} that
\begin{align*}
    \P\Big(\sup_{a\in[n^{-1/2},M]}\Big|\Tilde{R}^{(n)}_{\an}-\bm{c}_{\infty}Z^{(n)}_{\an}\Big|>n\varepsilon\Big)\underset{n\to\infty}{\longrightarrow}0.
\end{align*}
Finally, recall the definition of $\mathfrak{S}_n$ in \eqref{EnsSingleExcu} and
\begin{align*}
    \P\Big(\sup_{a\in[n^{-1/2},M]}\Big|R^{(n)}_{\an}-\bm{c}_{\infty}Z^{(n)}_{\an}\Big|>n\varepsilon\Big)\leq\P\Big(\sup_{a\in[n^{-1/2},M]}\Big|\Tilde{R}^{(n)}_{\an}-\bm{c}_{\infty}Z^{(n)}_{\an}\Big|>n\varepsilon\Big)+1-\P(\mathfrak{S}_n),
\end{align*}
and Lemma \ref{SingleExcu} with $\gamma_n=\PEsqrt$ yields the first part of Lemma \ref{ReducedProcesses}. For the second part, using similar arguments, we obtain
\begin{align*}
    \sup_{p\leq\tn}\sup_{a\in[n^{-1/2},M]}\E\Big[\frac{1}{n^2}\Big|Z^{(n)}_{ \an+1}-Z^{(n)}_{\an}\Big|^2\Big]\underset{n\to\infty}{\longrightarrow}0.
\end{align*}
Recall that $L^{(n)}_k=Z^{(n)}_k+Z^{(n)}_{k+1}$ so this convergence, together with Lemma \ref{InegMax} allow to complete the proof.
\end{proof}

\begin{lemm}\label{Conv1ExcuUnif}
Let $i\geq 1$ and $(a_{i,n})_{n\geq 1}$ be a sequence of non-negative integers such that $\sup_{i\geq 1}a_{i,n}\leq a_n$ where $a_n/n\to0$ as $n\to\infty$. For any $a,\varepsilon>0$
\begin{enumerate}[label=(\roman*)]
    \item\label{Conv1ExcuUnif1}
        \begin{align*}
            \lim_{n\to\infty}\sup_{i\geq 1}n\P\Big(\big|R^{(1)}_{\an}-\bm{c}_{\infty}Z^{(1)}_{\an}\big|>n\varepsilon\Big)=0;
        \end{align*}
    \item\label{Conv1ExcuUnif2}
        \begin{align*}
            \lim_{n\to\infty}\sup_{i\geq1}n\P\Big(\big|Z^{(1)}_{\an-a_{i,n}}-Z^{(1)}_{\an}\big|>n\varepsilon\Big)=0.
        \end{align*}
\end{enumerate}
\end{lemm}
\noindent The proof of Lemma \ref{Conv1ExcuUnif} is very similar to the one of Lemma \ref{ReducedProcesses}. Indeed, using the same arguments together with the fact that $\sup_{i\geq 1}a_{i,n}\leq a_n$, one can prove that
\begin{align*}
    \lim_{n\to\infty}\sup_{i\geq 1}n^2\E\Big[\Big|R^{(1)}_{\an-a_{i,n}}-\bm{c}_{\infty}Z^{(1)}_{\an-a_{i,n}}\Big|^2\Big]=0\;\textrm{ and }\;\lim_{n\to\infty}\sup_{i\geq 1}n^2\E\Big[\Big|Z^{(1)}_{\an-a_{i,n}}-Z^{(1)}_{\an}\Big|^2\Big]=0
\end{align*}

\noindent We are now ready to prove Theorem \ref{ThMultiMar}.

\begin{proof}[Proof of Theorem \ref{ThMultiMar}]

\noindent Recall that we aim to prove that the sequence $((\frac{1}{n}\VectCoord{R}{n}_{\an},\;\frac{1}{n}\VectCoord{Z}{n}_{\an},\;\frac{1}{n}\VectCoord{L}{n}_{\an};\;  a\geq 0))_{n\geq 1}$ converges in law, under $\P$, for the Skorokhod topology on $D([0,\infty),\R^3)$ towards the random process $(\bm{c}_{\infty}\mathcal{Z}_a(W_{\infty}),\; \mathcal{Z}_a(W_{\infty}),\;2\mathcal{Z}_a(W_{\infty});\; a\geq 0)$ where we recall that $\mathcal{Z}_a(t)=t\mathcal{Y}_{a/t}$ if $t>0$ and $\mathcal{Z}_a(0)=0$. For that, we proceed in two distinct steps.

\vspace{0.2cm}

\noindent\textbf{Step 1}: tightness of the \textit{multi-type additive martingale}

\vspace{0.1cm}

\noindent The first purpose of this step is to prove that $((\frac{1}{n}\VectCoord{Z}{n}_{\an};\;  a\geq 0))_{n\geq 1}$ is tight in $D([0,\infty),\R)$ under $\P$. We actually prove that it is $C$-tight in the sense of Definition 3.25, Chapter VI in \cite{JacodShiryaev}. For that, it will suffice to show the following: for any $M>0$
\begin{align}\label{L2Bounded}
    \sup_{n\geq 1}\E\Big[\sup_{a\in[0,M]}\big(Z^{(n)}_{\an}/n\big)^2\Big]<\infty.
\end{align}
and for any $\varepsilon>0$
\begin{align}\label{Ctight2}
        \lim_{\delta\to 0}\limsup_{n\to\infty}\P\Big(\sup_{a,b\in[0,M];\; 0<a-b<\delta}\big|Z^{(n)}_{\an}-Z^{(n)}_{\bn}\big|>\varepsilon n\Big)=0.
\end{align}
For \eqref{L2Bounded}, thanks to the Doob inequality
\begin{align*}
    \E\Big[\sup_{a\in[0,M]}\big(Z^{(n)}_{\an}/n\big)^2\Big]\leq 4\E\Big[\big(Z^{(n)}_{\Mn}(\tn)/n\big)^2\Big].
\end{align*}
Using similar arguments as in the proof of Lemma \ref{ReducedProcesses}, one can see that for any $n\geq 1$
\begin{align*}
    \E\Big[\Big(\frac{1}{n}Z^{(n)}_{\Mn}\Big)^2\Big]\leq\frac{2}{n(1-e^{\psi(2)})}+2ac_0+1,
\end{align*}
and deduce \eqref{L2Bounded}. For \eqref{Ctight2}, we have, thanks to Lemma \ref{InegMax} in a first time and Doob inequality to the martingale $(Z^{(n)}_{\ell+q}-Z^{(n)}_{\ell})_{q\geq 0}$ in a second time that for any $n\geq 2/\delta$
\begin{align*}
    &\P\Big(\sup_{a,b\in[0,M];\; 0<a-b<\delta}\big|Z^{(n)}_{\an}-Z^{(n)}_{\bn}(p)\big|>\varepsilon n\Big) \\ & \leq\frac{1}{(\varepsilon n)^2}\sup_{\ell\leq\Mn}\E\Big[\sup_{q\leq\lfloor 2\delta n\rfloor}\big(Z^{(n)}_{\ell+q}-Z^{(n)}_{\ell}\big)^2\Big]\leq\frac{4}{(\varepsilon n)^2}\sup_{\ell\leq\Mn}\E\Big[\big(Z^{(n)}_{\ell+\lfloor 2\delta n\rfloor}-Z^{(n)}_{\ell}\big)^2\Big].
\end{align*}
Note that $\E[(Z^{(n)}_{\ell+\lfloor 2\delta n\rfloor}-Z^{(n)}_{\ell})^2]=\E[(Z^{(n)}_{\ell+\lfloor 2\delta n\rfloor})^2]-\E[(Z^{(n)}_{\ell})^2]$ and similarly as is the proof of Lemma \ref{ReducedProcesses}, we have, as $n\to\infty$, for any $\ell\leq\Mn$ and $p\leq\tn$
\begin{align*}
    \E\Big[\Big(Z^{(n)}_{\ell+\lfloor 2\delta n\rfloor}\Big)^2\Big]=o(n^2)+2nc_0(1-e^{\psi(2)})\sum_{k<\ell+\lfloor 2\delta n\rfloor}\;\sum_{j\leq k}e^{j\psi(2)}+p^2,
\end{align*}
and
\begin{align*}
    \E\Big[\Big(Z^{(n)}_{\ell}\Big)^2\Big]=o(n^2)+2nc_0(1-e^{\psi(2)})\sum_{k<\ell}\;\sum_{j\leq k}e^{j\psi(2)}+p^2.
\end{align*}
It yields
\begin{align*}
    \E\Big[\Big(Z^{(n)}_{\ell+\lfloor 2\delta n\rfloor}-Z^{(n)}_{\ell}\Big)^2\Big]&=o(n^2)+2nc_0(1-e^{\psi(2)})\sum_{k=\ell}^{\ell+\lfloor 2\delta n\rfloor-1}\sum_{j\leq k}e^{j\psi(2)}\leq o(n^2)+2\delta c_0n^2,
\end{align*}
that is, for any $\delta>0$
\begin{align*}
    \limsup_{n\to\infty}\P\Big(\sup_{a,b\in[0,M];\; 0<a-b<\delta}\big|Z^{(n)}_{\an}-Z^{(n)}_{\bn}\big|>\varepsilon n\Big)\leq\frac{8\delta c_0}{\varepsilon^2},
\end{align*}
thus giving \eqref{Ctight2}.

\vspace{0.2cm}

\noindent\textbf{Step 2}: convergence in law under $\P$

\vspace{0.1cm}

\noindent In this step, we prove that in law under $\P$, the sequence $((\frac{1}{n}\VectCoord{R}{n}_{\an},\;\frac{1}{n}\VectCoord{Z}{n}_{\an},\;\frac{1}{n}\VectCoord{L}{n}_{\an};\;  a\geq 0))_{n\geq 1}$ converges for the Skorokhod topology on $D([0,\infty),\R^3)$ towards $(\bm{c}_{\infty}\mathcal{Z}_a(W_{\infty}),\; \mathcal{Z}_a(W_{\infty}),\;2\mathcal{Z}_a(W_{\infty});\; a\geq 0)$. Thanks to Lemma \ref{ReducedProcesses}, we only have to prove that $((\frac{1}{n}\VectCoord{L}{n}_{\an};\;  a\geq 0))_{n\geq 1}$ converges to $(2\mathcal{Z}_a(W_{\infty});\; a\geq 0)$. For that, we claim that under $\P$, in law on $\R\times D([0,\infty),\R)$
\begin{align}\label{JointConvClaim}
    \left(\frac{C_{\infty}}{n^2}\tau^n\,;\Big(\frac{C_{\infty}}{n}\big|X_{\lfloor sn^2\rfloor}\big|;\; s\geq 0\Big)\right)\underset{n\to\infty}{\longrightarrow}\Big(\tau_{-W_{\infty}},\;\big(H_{sC_{\infty}};\; s\geq 0\big)\Big),
\end{align}
where $Y_s=(2c_0/C_{\infty})B_s$, with $(B_s;\; s\geq 0)$ a standard Brownian motion, $\tau_{-W_{\infty}}:=\inf\{s\geq 0;\; Y_s=-W_{\infty}\}$ and $(H_s;\; s\geq 0)$ is the continuous-time height process associated with $(Y_s;\; s\geq 0)$, see Definition 1.2.1 in \cite{DuqLeGall}. Note that $(H_s;\; s\geq 0)$ is distributed as $((2C_{\infty}/c_0)^{1/2}|B_s|;\; s\geq 0)$. The convergence of $(C_{\infty}\tau^n/n^2)_{n\geq 1}$ to $\tau_{-W_{\infty}}$ has been proven in \cite{Hu2017} (Corollary 1.2) and in \cite{AK23LocalTimes} (Theorem 1.2). The convergence of $((C_{\infty}\big|X_{\lfloor sn^2\rfloor}/n\big|;\; s\geq 0))_{n\geq 1}$ to $(H_{sC_{\infty}};\; s\geq 0)$ has been proven in \cite{AidRap} (Theorem 6.1). The fact that these two convergences hold jointly is a consequence of Corollary 2.5.1 in \cite{DuqLeGall}, using similar arguments as in Fact 3.4 in \cite{AK24LocCrtiGen}. \\
We can now use \eqref{JointConvClaim} to prove the convergence of the rescaled local time. For that, we follow an argument used in the proof of Corollary 2.5.1 in \cite{DuqLeGall}. Thanks to \textbf{Step 2}, \eqref{JointConvClaim} and the Skorokhod
representation theorem, one can assume, without loss of generality, that the converge
\begin{align}\label{JointConvAS}
    \left(\frac{C_{\infty}}{n^2}\tau^n\,;\Big(\frac{C_{\infty}}{n}\big|X_{\lfloor sn^2\rfloor}\big|;\; s\geq 0\Big);\;\Big(\frac{1}{n}L^{(n)}_{\an};\; a\geq 0\Big) \right)\underset{n\to\infty}{\longrightarrow}\Big(\tau_{-W_{\infty}},\;\big(H_{sC_{\infty}};\; s\geq 0\big);\; \big(L_a;\; a\geq 0\big)\Big) 
\end{align}
holds $\P$-almost surely and along a sub-sequence, for some càdlàg process $(L_a;\; a\geq 0)$. Let $\varphi:[0,\infty)\to[0,\infty)$ be a Lipschitz continuous function with compact support. Thanks to Lemma \ref{InegMax} and \eqref{L2Bounded}, we have that $\sup_{n\geq 1}\sup_{a\in[0,M]}L^{(n)}_{\an}/n$ is finite for all $M>0$, $\P$-almost surely, so on the one hand, \eqref{JointConvAS} and the dominated convergence theorem yields, $\P$-almost surely
\begin{align}\label{ConvIntRange}
    \frac{1}{n}\int_0^{\infty}\varphi(a)L^{(n)}_{\an}\mathrm{d}a\underset{n\to\infty}{\longrightarrow}\int_0^{\infty}\varphi(a)L_a\mathrm{d}a.
\end{align}
On the other hand, by definition
\begin{align*}
    L^{(n)}_{\an}=\sum_{j=0}^{\tau^{n}-1}\un_{\{|X_j|=\an\}}=\sum_{j=0}^{\tau^n-1}\un_{\big[|X_j|/n,(|X_j|+1)/n)\big)}(a).
\end{align*}
Hence 
\begin{align*}
    \frac{1}{n}\int_0^{\infty}\varphi(a)L^{(n)}_{\an}\mathrm{d}a&=\Delta_n+\frac{1}{n^2}\sum_{j=0}^{\tau^n-1}\varphi\Big(\frac{|X_j|}{n}\Big) \\ & =\Delta_n+\frac{1}{n^2}\tau^n\int_0^1\varphi\Big(\frac{|X_{\lfloor s\tau^n\rfloor}|}{n}\Big)\mathrm{d}s,
\end{align*}
where 
\begin{align*}
    \Delta_n:=\frac{1}{n}\sum_{j=0}^{\tau^n-1}\int_{|X_j|/n}^{(|X_j|+1)/n}\Big(\varphi(a)-\varphi\big(|X_j|)/n\big)\Big)\mathrm{d}a.
\end{align*}
One can see that
\begin{align*}
    |\Delta_n|\leq\frac{\mathfrak{C}_{\varphi}}{n}\sum_{j=0}^{\tau^n-1}\int_{|X_j|/n}^{(|X_j|+1)/n}\Big(a-\frac{|X_j|}{n}\Big)\mathrm{d}a\leq\frac{\mathfrak{C}_{\varphi}}{n^3}\tau^n,
\end{align*}
for some positive constant $\mathfrak{C}_{\varphi}$ coming from the fact that $\varphi$ is Lipschitz continuous. By \eqref{JointConvAS}, $\P$-almost surely, $\Delta_n\to0$ as $n\to\infty$ and
\begin{align*}
     \frac{1}{n}\int_0^{\infty}\varphi(a)\VectCoord{L}{n}_{\an}\mathrm{d}a\underset{n\to\infty}{\longrightarrow}&\frac{2}{C_{\infty}}\tau_{-W_{\infty}}\int_0^1\varphi\big(H_{s\tau_{-W_{\infty}}}/C_{\infty}\big)\mathrm{d}s \\[0.7em] & =\frac{2}{C_{\infty}}\int_0^{\tau_{-W_{\infty}}}\varphi\big(H_s/C_{\infty}\big)\mathrm{d}s.
\end{align*}
By a representation theorem for local times (see Proposition 1.3.3 in \cite{DuqLeGall}), we have
\begin{align*}
    \int_0^{\tau_{-W_{\infty}}}\varphi\big(H_s/C_{\infty}\big)\mathrm{d}s=\int_0^{\infty}\varphi(a/C_{\infty})L(\tau_{-W_{\infty}},a)\mathrm{d}a=C_{\infty}\int_0^{\infty}\varphi(a)L(\tau_{-W_{\infty}},aC_{\infty})\mathrm{d}a,
\end{align*}
where $L(\tau_{-W_{\infty}},a)$ is the local time of $H$ at time $\tau_{-W_{\infty}}$ and level $a$. Hence, it follows from \eqref{ConvIntRange} that 
\begin{align*}
    \int_0^{\infty}\varphi(a)L_a\mathrm{d}t=2\int_0^{\infty}\varphi(a)L(\tau_{-W_{\infty}},aC_{\infty})\mathrm{d}a,
\end{align*}
and then $L_a=2L(\tau_{-W_{\infty}},aC_{\infty})$. In particular, the limit of $((\VectCoord{L}{n}_{\an}/n;\; a\geq 0))_{n\geq 1}$ does not depend on the choice of the sub-sequence and equals $(2L(\tau_{-W_{\infty}},aC_{\infty});\; a\geq 0)$. Thanks to the Ray-Knight theorem (see Theorem 1.4.1 in \cite{DuqLeGall}), the process $(L(\tau_{-W_{\infty}},a);\; a\geq 0)$ is a continuous-state branching process with branching mechanism $\lambda\mapsto\frac{c_0}{C_{\infty}}\lambda^2$ starting from $W_{\infty}$, that is $L(\tau_{-W_{\infty}},0)=W_{\infty}$, which, thanks to \eqref{LaplaceCSBP}, is distributed as $W_{\infty}\mathcal{Y}_{a/(C_{\infty}W_{\infty})}$, thus giving that $(2L(\tau_{-W_{\infty}},aC_{\infty});\; a\geq 0)$ is distributed as $(2\mathcal{Z}_{a}(W_{\infty});\; a\geq 0)$. The proof of the theorem is now completed.
\end{proof}

\noindent Let us now prove Theorem \ref{ThMultiMar}.

\subsection{ Proof of Theorem \ref{GenealogyDiffCritic}}\label{SectionGenealogyDiffCritic}

Before proving Theorem \ref{GenealogyDiffCritic}, we need the following result:
\begin{lemm}\label{ConvLoiCondi}
Let $f:\R\to\R$ be a continuous and non-negative function and $k\geq 1$. Let $g_k(n,0,f)=e^{-f(0)}$ and for any $p\geq 1$, define 
\begin{align}\label{Defg}
    g_k(n,p,f):=\E\Big[e^{-f(\frac{1}{n}R^{(p)}_k)}\Big].
\end{align}
For any $a>0$
\begin{align*}
    \lim_{n\to\infty}n\big(1-g_{\an}(n,1,f)\big)= \frac{1}{c_0a}\Big(1-\int_0^{\infty}e^{-f(\bm{c}_{\infty}c_0at)-t}\mathrm{d}t\Big),
\end{align*}
\end{lemm}
\begin{proof}
Let us prove that the sequence of probability measure $\P(\frac{1}{n}R^{(1)}_{\an}\in\cdot|R^{(1)}_{\an}>0)$ converges to the law of $\bm{c}_{\infty}c_0a\xi$ where $\xi$ is an Exponential random variable with mean $1$.
For that, define, for any $n\geq 1$
\begin{align*}
    \mathcal{B}^{(n)}:=\big\{x\in\T,\; |x|>\lfloor(\log n)^2\rfloor;\; N_x^{(n)}=1\textrm{ and }\min_{\ell<i<|x|}N_{x_i}^{(n)}\geq 2\big\}.
\end{align*}
This set was first introduced by E. Aïdékon and L. de Raphélis in \cite{AidRap}. Thanks to Lemma 7.2 in \cite{AidRap}, $\lim_{n\to\infty}\P^{\mathcal{E}}(\forall\; z\in\mathcal{B}^{(n)};\;\lfloor(\log n)^2\rfloor\leq|z|\leq\lfloor(\log n)^3\rfloor;\;\forall x\in\T,\; |x|=\an:\; \exists\; z\in\mathcal{B}^{(n)}\textrm{ such that }z<x)=1$, $\Pb$-almost surely. In particular, $\Pb$-almost surely
\begin{align}\label{ReducedMultiMart}
    \P^{\mathcal{E}}\left(Z^{(n)}_{\an}=\sum_{z\in\mathcal{B}^{(n)}}\sum_{x>z}N^{(n)}_x\un_{\{|x|_z=\an-|z|\}}\right)\underset{n\to\infty}{\longrightarrow}1,
\end{align}
with the convention $\sum_{\varnothing}=0$ and $|x|_z$ denotes the generation of the vertex $x$ with respect to the tree rooted at $z$ and made up of descendent of $z$. This reduce multi-type additive martingale is nice to deal with because unlike $Z_{\an}$, we have that under the annealed probability $\P$ and conditionally given $\mathcal{B}^{(n)}$, the random variables $(\sum_{x>z}N^{(n)}_x\un_{\{|x|_z=\an-|z|\}},\; z\in\mathcal{B}^{(n)})$ are independent copies of $(Z^{(1)}_{\an-|z|},\; z\in\mathcal{B}^{(n)})$. \\
Now, on the one hand, by Theorem \ref{ThMultiMar} and \eqref{LaplaceCSBP}, we have
\begin{align}\label{LimLaplceMultiMart}
    \lim_{n\to\infty}\E\Big[e^{-\frac{\theta}{n}Z^{(n)}_{\an}}\Big]=\Eb\Big[e^{-\frac{\theta W_{\infty}}{1+\theta c_0a}}\Big].
\end{align}
One the other hand, by \eqref{ReducedMultiMart}
\begin{align*}
    \lim_{n\to\infty}\E\Big[e^{-\frac{\theta}{n}Z^{(n)}_{\an}}\Big]=\lim_{n\to\infty}\E\left[\prod_{z\in\mathcal{B}^{(n)}}e^{-\frac{\theta}{n}\sum_{x>z}N^{(n)}_x\un_{\{|x|_z=\an-|z|\}}}\right].
\end{align*}
By definition
\begin{align*}
    \E\left[\prod_{z\in\mathcal{B}^{(n)}}e^{-\frac{\theta}{n}\sum_{x>z}N^{(n)}_x\un_{\{|x|_z=\an-|z|\}}}\bigg|\mathcal{B}^{(n)}\right]=\prod_{z\in\mathcal{B}^{(n)}}\E\Big[e^{-\frac{\theta}{n}Z^{(1)}_{\an-|z|}}\Big].
\end{align*}
We know thanks to Lemma 2.9 in \cite{AK23LocalTimes} that in $\P$-probability, $\CardRoots/n\to W_{\infty}$ as $n\to\infty$. Hence, by Lemma \ref{Conv1ExcuUnif} \ref{Conv1ExcuUnif2} 
\begin{align*}
    \lim_{n\to\infty}\E\left[\prod_{z\in\mathcal{B}^{(n)}}e^{-\frac{\theta}{n}\sum_{x>z}N^{(n)}_x\un_{\{|x|_z=\an-|z|\}}}\right]=\lim_{n\to\infty}\E\left[\E\Big[e^{-\frac{\theta}{n}Z^{(1)}}\Big]^{\CardRoots}\right]=\Eb\left[\Big(\lim_{n\to\infty}\E\Big[e^{-\frac{\theta}{n}Z^{(1)}}\Big]^n\Big)^{W_{\infty}}\right].
\end{align*}
Recalling that $\Pb$-almost surely, $\{W_{\infty}>0\}=\{\textrm{non-extinction of }\T\}$, we have necessarily $\lim_{n\to\infty}n(1-\E[e^{-\frac{\theta}{n}Z^{(1)}_{\an}}])=\theta/(1+\theta c_0a)$ by \eqref{LimLaplceMultiMart}. Then, using that $\lim_{n\to\infty}n\P(R^{(1)}_{\an}>0)=c_0/a$ (see Theorem 1.2 in \cite{RousselinConduc} or Theorem 2.1 in \cite{AK24LocCrtiGen}), we obtain that the sequence of probability measures $(\P(Z^{(1)}_{\an}/n\in\cdot\;|Z^{(1)}_{\an}>0))_{n\geq 1}$ converges to the law of $c_0a\xi$. Moreover, for any $\varepsilon>0$
\begin{align*}
    \P\Big(\Big|\frac{1}{n}R^{(1)}_{\an}-\frac{\bm{c}_{\infty}}{n}Z^{(1)}_{\an}\Big|>\varepsilon\Big|Z^{(1)}_{\an}>0\Big)=\frac{1}{\P\big(Z^{(1)}_{\an}>0\big)}\P\Big(\Big|\frac{1}{n}R^{(1)}_{\an}-\frac{\bm{c}_{\infty}}{n}Z^{(1)}_{\an}\Big|>\varepsilon\Big).
\end{align*}
which goes to $0$ as $n\to\infty$ by Lemma \ref{Conv1ExcuUnif} \ref{Conv1ExcuUnif1}. Hence, we get that the sequence of probability measures $(\P(R^{(1)}_{\an}/n\in\cdot\;|R^{(1)}_{\an}>0))_{n\geq 1}$ converges to the law of $\bm{c}_{\infty}c_0a\xi$ and the proof is completed.
\end{proof}

\begin{proof}[Proof of Theorem \ref{GenealogyDiffCritic}]
The proof is partially inspired from the one of Theorem 2.1 in \cite{Athreya_SUB_CT} for (regular) critical Galton-Watson trees. First, we aim to prove \eqref{RecentPastCoa}, that is, for any $0<a<b$
\begin{align*}
    \lim_{n\to\infty}\P\big(|\VectCoord{\mathcal{X}}{1,n}_{\bn}\land\VectCoord{\mathcal{X}}{2,n}_{\bn}|\geq\an\big)=\E\left[\frac{\sum_{j=1}^{N_{a,b}}(\xi_j)^2}{(\sum_{j=1}^{N_{a,b}}\xi_j)^2}\Big| N_{a,b}>0\right],
\end{align*}
with $N_{a,b}=\sum_{i=1}^{N_b}G^i_{a,b}$, where $N_b$ is, under $\P^{\mathcal{E}}$, a Poisson random variable with parameter $W_{\infty}/(bc_0)$, $(G^i_{a,b})_{i\geq 1}$ is a sequence of\textrm{ i.i.d }Geometric random variables on $\N^*$ with probability of success $1-a/b$. $(\xi_j)_{j\geq 1}$ is a sequence of\textrm{ i.i.d }Exponential random variables with mean $1$. Besides, all random variables involved are independent. \\
Introduce, for any $p\geq 1$ and $1\leq k'\leq k$ 
\begin{align*}
    \mathcal{C}^{(p)}_{k,k'}:=\sum_{|x|=|y|=k}\un_{\{N_x^{(p)}\land N_y^{(p)}\geq 1,\; |x\land y|\geq k'\}},
\end{align*}
$\mathcal{C}^{(p)}_{k,k'}$ is the number of couples of vertices in generation $k$ of the tree $\mathcal{R}^{(p)}$ such that their most recent common ancestor is located in a generation larger than $k'$. Let us prove that for any $a\in(0,b)$, in law in $\P$
\begin{align}\label{ConvDoubleRange1}
    \Big(\Big(\frac{1}{n^2}\mathcal{C}^{(n)}_{\bn,\an},\; \frac{1}{n}R^{(n)}_{\bn}\Big)\Big|\; R^{(n)}_{\bn}>0\Big)\underset{n\to\infty}{\longrightarrow}\left(\Bigg(\sum_{j=1}^{N_{a,b}}\big(\bm{c}_{\infty}c_0(b-a)\xi_j\big)^2,\;\sum_{j=1}^{N_{a,b}}\bm{c}_{\infty}c_0(b-a)\xi_j\Bigg)\Bigg|\; N_{a,b}>0\right).
\end{align}
Define
\begin{align*}
    \mathbfcal{C}^{(n,p)}_{k,k'}(f):=\sum_{|u|=k'}f\Bigg(\frac{1}{n}\underset{x>u}{\sum_{|x|=k}}\un_{\{N_x^{(p)}\geq 1\}}\Bigg)\;\textrm{ and }\;C^{(n)}_{k,k'}(f):=C^{(n,n)}_{k,k'}.
\end{align*}
One can notice that, on the event $\{R^{(n)}_{\bn}>0\}$
\begin{align*}
    \frac{1}{n^2}\mathcal{C}^{(n)}_{\bn,\an}=\mathbfcal{C}^{(n)}_{\bn,\an}(\mathfrak{z}\mapsto \mathfrak{z}^2)\;\;\textrm{ and }\;\;\frac{1}{n}R^{(n)}_{\bn}=\mathbfcal{C}^{(n)}_{\bn,\an}(\mathfrak{z}\mapsto\mathfrak{z})
\end{align*}
so in order to get \eqref{ConvDoubleRange1}, it is enough to prove that
\begin{align}\label{ConvDoubleRange2}
    \lim_{n\to\infty}\E\Big[e^{-\mathbfcal{C}^{(n)}_{\bn,\an}(f)}\Big|R^{(n)}_{\bn}>0\Big]=\E\Big[e^{-\sum_{j=1}^{N_{a,b}}f(\bm{c}_{\infty}c_0(b-a)\xi_j)}\Big|N_{a,b}>0\Big],
\end{align}
with $f(\mathfrak{z})=\lambda_1\mathfrak{z}^2+\lambda_2\mathfrak{z}$ for $\lambda_1,\lambda_2\geq 0$ given. First, one can notice that 
\begin{align*}
    \E\Big[e^{-\mathbfcal{C}^{(n)}_{\bn,\an}(f)}\Big|\;R^{(n)}_{\bn}>0\Big]&=\frac{\E\Big[e^{-\mathbfcal{C}^{(n)}_{\bn,\an}(f)}\Big]}{\P\big(R^{(n)}_{\bn}>0\big)}-\frac{\P\big(R^{(n)}_{\bn}=0\big)}{\P\big(R^{(n)}_{\bn}>0\big)}.
\end{align*}
Thanks to the branching property (see Fact \ref{FactGWMulti})
\begin{align*}
    \E\Big[e^{-\mathbfcal{C}^{(n)}_{\bn,\an}(f)}\Big] =\E\Big[\prod_{u\in\mathcal{R}^{n};|u|=\an}g_{\bn-\an}\big(n,N_u^{(n)},f\big)\Big],
\end{align*}
where $g_k(n,p,f)$ is defined in \eqref{Defg}, and we recall that $f(\mathfrak{z})=\lambda_1\mathfrak{z}^2+\lambda_2\mathfrak{z}$ for $\lambda_1,\lambda_2\geq 0$ given. We now decompose $g_{\bn-\an}(n,p,f)$ as follows: for any positive integer $p\leq n\log n$
\begin{align*}
    g_{\bn-\an}(n,p,f)&=\E\Big[e^{-f(\frac{1}{n}R^{(p)}_{\bn-\an})}\un_{\mathfrak{S}_n}\Big]+\E\Big[e^{-f(\frac{1}{n}R^{(p)}_{\bn-\an})}\big(1-\un_{\mathfrak{S}_n}\big)\Big] \\ & =\E\Big[e^{-f(\frac{1}{n}\Tilde{R}^{(p)}_{\bn-\an})}\un_{\mathfrak{S}_n}\Big]+\E\Big[e^{-f(\frac{1}{n}R^{(p)}_{\bn-\an})}\big(1-\un_{\mathfrak{S}_n}\big)\Big].
\end{align*}
Note that for any $p\geq 1$
\begin{align*}
    f\Big(\frac{1}{n}\Tilde{R}^{(p)}_{\bn-\an}\Big)=&\sum_{j=1}^p\Bigg(\frac{\lambda_1}{n^2}\Big(\sum_{|x|=\bn-\an}\un_{\{N_x^{(j)}-N_x^{(j-1)}\geq 1\}}\Big)^2+\frac{\lambda_2}{n}\sum_{|x|=\bn-\an}\un_{\{N_x^{(j)}-N_x^{(j-1)}\geq 1\}}\Bigg) \\ & +\frac{\lambda_1}{n^2}\sum_{i\not=j=1}^p\;\sum_{|x|=|y|=\bn-\an}\un_{\{N_x^{(i)}-N_x^{(i-1)}\geq 1,\; N_y^{(j)}-N_y^{(j-1)}\geq 1\}}.
\end{align*}
Let us first provide a lower bound for $\E[e^{-\mathbfcal{C}^{(n)}_{\bn,\an}(f)}]$. Recall that thanks to the strong Markov property, $(\sum_{|x|=\bn-\an}\un_{\{N_x^{(j)}-N_x^{(j-1)}\geq 1\}};\; 1\leq j\leq p)$ is a collection of $p$\textrm{ i.i.d }copies of $R^{(1)}_{\bn-\an}$ under $\P^{\mathcal{E}}$. Hence, by the FKG inequality, $\E^{\mathcal{E}}[e^{-f(\frac{1}{n}\Tilde{R}^{(p)}_{\bn-\an})}]$ is larger than
\begin{align*}
    &\E^{\mathcal{E}}\Big[e^{-\lambda_1\big(\frac{1}{n}R^{(1)}_{\bn-\an}\big)^2-\frac{\lambda_2}{n}R^{(1)}_{\bn-\an}}\Big]^p\E^{\mathcal{E}}\Big[e^{-\frac{\lambda_1}{n^2}\sum_{i\not=j=1}^p\;\sum_{|x|=|y|=\bn-\an}\un_{\{N_x^{(i)}-N_x^{(i-1)}\geq 1,\; N_y^{(j)}-N_y^{(j-1)}\geq 1\}}}\Big] \\ & \geq \E^{\mathcal{E}}\Big[e^{-f\big(\frac{1}{n}R^{(1)}_{\bn-\an}\big)}\Big]^pe^{-\frac{\lambda_1p^2}{n^2}\sum_{|x|=|y|=\bn-\an}\P^{\mathcal{E}}(N_x^{(1)}\geq 1)\P^{\mathcal{E}}(N_y^{(1)}\geq 1)},
\end{align*}
where we have used the Jensen inequality and the definition of $f$ for the last inequality. Using again that $\P^{\mathcal{E}}(N^{(1)}_z\geq 1)\leq e^{-V(z)}$, we obtain
\begin{align*}
    \E\Big[e^{-f(\frac{1}{n}\Tilde{R}^{(p)}_{\bn-\an})}\Big]\geq\Eb\left[\E^{\mathcal{E}}\Big[e^{-f\big(\frac{1}{n}R^{(1)}_{\bn-\an}\big)}\Big]^pe^{-\frac{\lambda_1p^2}{n^2}(W_{\bn-\an})^2}\right].
\end{align*}
Let $\delta_2:=\delta_1/4$ and $\delta_3:=\delta_1/(4+2\delta_1)$, where $\delta_1>0$ comes from Assumption \ref{Assumption1}. We have
\begin{align*}
    &\Eb\left[\E^{\mathcal{E}}\Big[e^{-f\big(\frac{1}{n}R^{(1)}_{\bn-\an}\big)}\Big]^pe^{-\frac{\lambda_1p^2}{n^2}(W_{\bn-\an})^2}\right] \\[0.7em] & \geq e^{-\frac{\lambda_1p^2}{n^{1+\delta_3}}}\Eb\left[\E^{\mathcal{E}}\Big[e^{-f\big(\frac{1}{n}R^{(1)}_{\bn-\an}\big)}\Big]^p\un_{\{(W_{\bn-\an})^{2+\delta_1}\leq n^{1+\delta_2}\}}\right] \\[0.7em] & \geq e^{-\frac{\lambda_1p^2}{n^{1+\delta_3}}}\Eb\left[\E^{\mathcal{E}}\Big[e^{-f\big(\frac{1}{n}R^{(1)}_{\bn-\an}\big)}\Big]\un_{\{(W_{\bn-\an})^{2+\delta_1}\leq n^{1+\delta_2}\}}\right]^p,
\end{align*}
where we have used again the Jensen inequality for the second inequality. Hence, $g_{\bn-\an}(n,p,f)$ is larger than 
\begin{align*}
    &e^{-\frac{\lambda_1p^2}{n^{1+\delta_3}}}\Eb\left[\E^{\mathcal{E}}\Big[e^{-f\big(\frac{1}{n}R^{(1)}_{\bn-\an}\big)}\Big]\un_{\{(W_{\bn-\an})^{2+\delta_1}\leq n^{1+\delta_2}\}}\right]^p-\big(1-\P(\mathfrak{S}_n)\big) \\[0.7em] & \geq e^{-\frac{\lambda_1p^2}{n^{1+\delta_3}}}\Eb\left[\E^{\mathcal{E}}\Big[e^{-f\big(\frac{1}{n}R^{(1)}_{\bn-\an}\big)}\Big]\un_{\{(W_{\bn-\an})^{2+\delta_1}\leq n^{1+\delta_2}\}}\right]^p-C_{\eqref{SingleExcu}}n^6e^{(\bn-\an)\psi(2)} \\[0.7em] & \geq e^{-\frac{\lambda_1p^2}{n^{1+\delta_3}}}\Big(g_{\bn-\an}(n,1,f)-\Pb\big((W_{\bn-\an})^{2+\delta_1}>n^{1+\delta_2}\big)\Big)^p-C_{\eqref{SingleExcu}}n^6e^{(\bn-\an)\psi(2)} \\[0.7em] & \geq e^{-\frac{\lambda_1p^2}{n^{1+\delta_3}}}\Big(g_{\bn-\an}(n,1,f)-\frac{\mathfrak{C}_{\infty}}{n^{1+\delta_2}}\Big)^p-C_{\eqref{SingleExcu}}n^6e^{(\bn-\an)\psi(2)},
\end{align*}
where we have used Lemma \ref{SingleExcu} for the second inequality and we have set $\mathfrak{C}_{\infty}:=\sup_{k\geq 0}\Eb[(W_k)^{2+\delta_1}]\in(0,\infty)$, which is well defined by Assumption \ref{Assumption1}. Note that for $n$ large enough and any $p\leq n\log n$
\begin{align}\label{Ingp1}
    e^{-\frac{\lambda_1p^2}{n^{1+\delta_3}}}\Big(g_{\bn-\an}(n,1,f)-\frac{\mathfrak{C}_{\infty}}{n^{1+\delta_2}}\Big)^p\geq C_{\eqref{SingleExcu}}n^8e^{(\bn-\an)\psi(2)}.
\end{align}
Indeed, we know thanks to Lemma \ref{ConvLoiCondi} that
\begin{align}\label{ConvLaplacef}
    \lim_{n\to\infty}n\big(1-g_{\bn-\an}(n,1,f)\big)=\frac{1-\phi_{a,b}(f)}{c_0(b-a)},
\end{align}
where $\phi_{a,b}(f):=\int_0^{\infty}e^{-f(\bm{c}_{\infty}c_0(b-a)t)-t}\mathrm{d}t$. In particular, for $n$ large enough, we have $g_{\bn-\an}(n,1,f)\geq 1-\frac{1}{2nc_0(b-a)}(1-\phi_{a,b}(f))$ and clearly, $\frac{1}{2nc_0(b-a)}(1-\phi_{a,b}(f))\geq\frac{\mathfrak{C}_{\infty}}{n^{1+\delta_2}}$ so for any $p\leq n\log n$ and $n$ large enough, we have
\begin{align*}
    e^{-\frac{\lambda_1p^2}{n^{1+\delta_3}}}\Big(g_{\bn-\an}(n,1,f)-\frac{\mathfrak{C}_{\infty}}{n^{1+\delta_2}}\Big)^p&\geq e^{-n^{1-\delta_2}(\log n)^2}\Big(1-\frac{1}{4nc_0}\big(1-\phi_{a,b}(f)\big)\Big)^{n\log n} \\ & \geq C_{\eqref{SingleExcu}}n^8e^{(\bn-\an)\psi(2)}.
\end{align*}
In particular, we deduce from \eqref{Ingp1} that for $n$ large enough and $1\leq p\leq n\log n$
\begin{align*}
    g_{n,a}(p,f)&\geq e^{-\frac{\lambda_1p^2}{n^{1+\delta_3}}}\Big(g_{\bn-\an}(n,1,f)-\frac{\mathfrak{C}_{\infty}}{n^{1+\delta_2}}\Big)^p\big(1-n^{-2}\big),
\end{align*}
and then $\E[e^{-\mathbfcal{C}^{(n)}_{\bn,\an}(f)}]$ is larger than
\begin{align}\label{EspBorneInf2}
    &\E\Big[e^{-\lambda_1\sum_{|u|=\an}\frac{(N^{(n)}_u)^2}{n^{1+\delta_2}}}\Big(g_{\bn-\an}(n,1,f)-\frac{\mathfrak{C}_{\infty}}{n^{1+\delta_2}}\Big)^{Z^{(n)}_{\an}}\big(1-n^{-2}\big)^{R^{(n)}_{\an}}\un_{\{\max_{|z|=\an}N^{(n)}_z\leq n\log n\}}\Big]\nonumber \\[0.7em] & \geq\E\Big[e^{-\lambda_1\sum_{|u|=\an}\frac{(N^{(n)}_u)^2}{n^{1+\delta_2}}}\Big(g_{\bn-\an}(n,1,f)-\frac{\mathfrak{C}_{\infty}}{n^{1+\delta_2}}\Big)^{Z^{(n)}_{\an}}\big(1-n^{-2}\big)^{Z^{(n)}_{\an}}\Big]-\P\Big(\max_{|z|=\an}N^{(n)}_z>n\log n\Big).
\end{align}
As we have already mentioned, we have $\E[\sum_{|u|=\an}(N^{(n)}_u)^2]\leq\mathfrak{C}_an$ for some constant $\mathfrak{C}_a>0$ and $n$ large enough. In particular, $\lim_{n\to\infty}\E[\sum_{|u|=\an}(N^{(n)}_u)^2]/n^{1+\delta_2}=0$. Moreover, since $\E[\sum_{|z|=\an}N^{(n)}_z]=n$, we deduce that $\lim_{n\to\infty}\P(\max_{|z|=\an}N^{(n)}_z>n\log n)=0$. Therefore, thanks to Theorem \ref{ThMultiMar}, claiming in particular that $(Z^{(n)}_{\an}/n)_{n\geq 1}$ converges in law under $\P$ to $W_{\infty}\mathcal{Y}_{a/W_{\infty}}$ and together with \eqref{ConvLaplacef}, we finally obtain
\begin{align}\label{liminfLaplace}
    \liminf_{n\to\infty}\E\Big[e^{-\mathbfcal{C}^{(n)}_{\bn,\an}(f)}\Big]\geq\E\Big[e^{-W_{\infty}\mathcal{Y}_{a/W_{\infty}}\frac{1-\phi_{a,b}(f)}{c_0(b-a)}}\Big]=\Eb\Big[e^{-\frac{W_{\infty}(1-\phi_{a,b}(f))}{c_0(b-a\phi_{a,b}(f))}}\Big], 
\end{align}
where we have used \eqref{LaplaceCSBP} for the equality. \\
Let us now provide an upper bound for $\E[e^{-\mathbfcal{C}^{(n)}_{\bn,\an}(f)}]$. Clearly
\begin{align*}
    \E\Big[e^{-\mathbfcal{C}^{(n)}_{\bn,\an}(f)}\Big]\leq&\E\Big[e^{-\sum_{j=1}^p\big(\frac{\lambda_1}{n^2}\big(\sum_{|x|=\bn-\an}\un_{\{N_x^{(j)}-N_x^{(j-1)}\geq 1\}}\big)^2+\frac{\lambda_1}{n}\sum_{|x|=\bn-\an}\un_{\{N_x^{(j)}-N_x^{(j-1)}\geq 1\}}\big)}\Big] \\ & =\Eb\left[\E^{\mathcal{E}}\Big[e^{-f(\frac{1}{n}R^{(1)}_{\bn-\an})}\Big]^p\right].
\end{align*}
Using the convexity of the exponential first and the fact that $e^{-t}\leq 1-t+t^2$ for any $t\geq 0$ in a second time, we have
\begin{align*}
    \E^{\mathcal{E}}\Big[e^{-f(\frac{1}{n}R^{(1)}_{\bn-\an})}\Big]^p&\leq e^{-p\big(1-\E^{\mathcal{E}}\big[e^{-f(\frac{1}{n}R^{(1)}_{\bn-\an})}\big]\big)} \\ & \leq 1-p\Big(1-\E^{\mathcal{E}}\big[e^{-f(\frac{1}{n}R^{(1)}_{\bn-\an})}\big]\Big)+p^2\Big(1-\E^{\mathcal{E}}\big[e^{-f(\frac{1}{n}R^{(1)}_{\bn-\an})}\big]\Big)^2,
\end{align*}
thus giving 
\begin{align*}
    \E\Big[e^{-f(\frac{1}{n}\Tilde{R}^{(p)}_{\bn-\an})}\Big]\leq1-p\big(1-g_{\bn-\an}(n,1,f))+p^2\Eb\Big[\Big(1-\E^{\mathcal{E}}\Big[e^{-f(\frac{1}{n}R^{(1)}_{\bn-\an})}\Big]\Big)^2\Big]    
\end{align*}
that is, using again the convexity of the exponential
\begin{align*}
     g_{\bn-\an}(n,p,f)&\leq e^{-p\big(1-g_{\bn-\an}(n,1,f)\big)}e^{p^2\Eb\big[\big(1-\E^{\mathcal{E}}\big[e^{-f(\frac{1}{n}R^{(1)}_{\bn-\an})}\big]\big)^2\big]}+\big(1-\P(\mathfrak{S}_n)\big) \\ & \leq e^{-p\big(1-g_{\bn-\an}(n,1,f)\big)}e^{p^2\Eb\big[\big(1-\E^{\mathcal{E}}\big[e^{-f(\frac{1}{n}R^{(1)}_{\bn-\an})}\big]\big)^2\big]}+C_{\eqref{SingleExcu}}n^6e^{(\bn-\an)\psi(2)},
\end{align*}
where we have used Lemma \ref{SingleExcu} for the second inequality. Note that for any $1\leq p\leq n\log n$ and $n$ large enough
\begin{align*}
    e^{-p\big(1-g_{\bn-\an}(n,1,f)\big)}e^{p^2\Eb\big[\big(1-\E^{\mathcal{E}}\big[e^{-f(\frac{1}{n}R^{(1)}_{\bn-\an})}\big]\big)^2\big]}&\geq e^{-n\log n\big(1-g_{\bn-\an}(n,1,f)\big)} \\ & \geq C_{\eqref{SingleExcu}}n^8e^{(\bn-\an)\psi(2)},
\end{align*}
where the last inequality comes from \eqref{ConvLaplacef}. Hence, for any $1\leq p\leq n\log n$ and $n$ large enough
\begin{align*}
     g_{\bn-\an}(n,p,f)\leq e^{-p\big(1-g_{\bn-\an}(n,1,f)\big)}e^{p^2\Eb\big[\big(1-\E^{\mathcal{E}}\big[e^{-f(\frac{1}{n}R^{(1)}_{\bn-\an})}\big]\big)^2\big]}\big(1+n^{-2}\big),
\end{align*}
thus giving that $\E[e^{-\mathbfcal{C}^{(n)}_{\bn,\an}(f)}]$ is smaller than
\begin{align*}
    &\E\Big[e^{-Z^{(n)}_{\an}\big(1-g_{\bn-\an}(n,1,f)\big)}e^{\sum_{|u|=\an}\big(N^{(n)}_u\big)^2\Eb\big[\big(1-\E^{\mathcal{E}}\big[e^{-f(\frac{1}{n}R^{(1)}_{\bn-\an})}\big]\big)^2\big]}\big(1+n^{-2}\big)^{Z^{(n)}_{\an}}\Big] \\ & +\P\Big(\max_{|u|=\an}N^{(n)}_u>n\log n\Big).
\end{align*}
Again, we have $\E[\sum_{|u|=\an}(N^{(n)}_u)^2]\leq\mathfrak{C}_an$ for some constant $\mathfrak{C}_a>0$ and $n$ large enough and this implies that $(\sum_{|u|=\an}(N^{(n)}_u)^2/n^2)_{n\geq 1}$ converges to $0$ in $\P$-probability. We now would like to prove that $(n^2\Eb[(1-\E^{\mathcal{E}}[e^{-f(\frac{1}{n}R^{(1)}_{\bn-\an})}])^2])_{n\geq 1}$ is bounded. Note that $f(0)=0$ so one can see that
\begin{align*}
    n\Big(1-\E^{\mathcal{E}}\Big[e^{-f(\frac{1}{n}R^{(1)}_{\bn-\an})}\Big]\Big)=n\P^{\mathcal{E}}\big(R^{(1)}_{\bn-\an}>0\big)\Big(1-\E^{\mathcal{E}}\Big[e^{-f(\frac{1}{n}R^{(1)}_{\bn-\an})}\Big|R^{(1)}_{\bn-\an}>0\Big]\Big),
\end{align*}
which yields
\begin{align*}
    \Eb\left[n^2\Big(1-\E^{\mathcal{E}}\Big[e^{-f(\frac{1}{n}R^{(1)}_{\bn-\an})}\Big]\Big)^2\right]\leq n^2\P\big(R^{(1)}_{\bn-\an}>0\big)^2\;\Eb\left[\left(\frac{\P^{\mathcal{E}}\big(R^{(1)}_{\bn-\an}>0\big)}{\P\big(R^{(1)}_{\bn-\an}>0\big)}\right)^2\right],
\end{align*}
which converges to $\Eb[(W_{\infty})^2]/c_0^2\in(0,\infty)$, see Theorem 1.2 in \cite{RousselinConduc} and Theorem 2.1 in \cite{AK24LocCrtiGen} more recently. Therefore, $(n^2\Eb[(1-\E^{\mathcal{E}}[e^{-f(\frac{1}{n}R^{(1)}_{\bn-\an})}])^2])_{n\geq 1}$ is bounded and then, in $\P$-probability
\begin{align*}
    \sum_{|u|=\an}\big(N_u^{(n)}\big)\Eb^*\left[n^2\Big(1-\E^{\mathcal{E}}\Big[e^{-f(\frac{1}{n}R^{(1)}_{\bn-\an})}\Big]\Big)^2\right]\underset{n\to\infty}{\longrightarrow}0.
\end{align*}
Hence, thanks to Theorem \ref{ThMultiMar}, claiming in particular that $(Z^{(n)}_{\an}/n)_{n\geq 1}$ converges in law under $\P$ to $W_{\infty}\mathcal{Y}_{a/W_{\infty}}$ and together with \eqref{ConvLaplacef}, we finally obtain
\begin{align*}
    \limsup_{n\to\infty}\E\Big[e^{-\mathbfcal{C}^{(n)}_{\bn,\an}(f)}\Big]\leq\E\Big[e^{-W_{\infty}\mathcal{Y}_{a/W_{\infty}}\frac{1-\phi_{a,b}(f)}{c_0(b-a)}}\Big]=\Eb\Big[e^{-\frac{W_{\infty}(1-\phi_{a,b}(f))}{c_0(b-a\phi_{a,b}(f))}}\Big],
\end{align*}
that is, by \eqref{liminfLaplace}
\begin{align*}
    \lim_{n\to\infty}\E\Big[e^{-\mathbfcal{C}^{(n)}_{\bn,\an}(f)}\Big]=\Eb\Big[e^{-\frac{W_{\infty}(1-\phi_{a,b}(f))}{c_0(b-a\phi_{a,b}(f))}}\Big].
\end{align*}
To complete the proof of \eqref{ConvDoubleRange2}, we are left to check that 
\begin{align*}
    \frac{\Eb\Big[e^{-\frac{W_{\infty}(1-\phi_{a,b}(f))}{c_0(b-a\phi_{a,b}(f))}}\Big]}{\P\big(R^{(n)}_{\bn}>0\big)}-\frac{\P\big(R^{(n)}_{\bn}=0\big)}{\P^*\big(R^{(n)}_{\bn}>0\big)}\underset{n\to\infty}{\longrightarrow}\E\Big[e^{-\sum_{j=1}^{N_{a,b}}f(\bm{c}_{\infty}c_0(b-a)\xi_j)}\Big|N_{a,b}>0\Big].
\end{align*}
We have
\begin{align*}
    \P\big(R^{(n)}_{\bn}>0\big)=1-\Eb\Big[\Big(1-\P^{\mathcal{E}}\big(R^{(1)}_{\bn}>0\big)\Big)^n\Big]\underset{n\to\infty}{\longrightarrow}\Eb\Big[1-e^{-\frac{W_{\infty}}{bc_0}}\Big].
\end{align*}
Hence
\begin{align*}
    \frac{\Eb\Big[e^{-\frac{W_{\infty}(1-\phi_{a,b}(f))}{c_0(b-a\phi_{a,b}(f))}}\Big]}{\P\big(R^{(n)}_{\bn}>0\big)}-\frac{\P\big(R^{(n)}_{\bn}=0\big)}{\P\big(R^{(n)}_{\bn}>0\big)}\underset{n\to\infty}{\longrightarrow}\frac{\Eb\Big[e^{-\frac{W_{\infty}(1-\phi_{a,b}(f))}{c_0(b-a\phi_{a,b}(f))}}\Big]}{\Eb\Big[1-e^{-\frac{W_{\infty}}{bc_0}}\Big]}-\frac{\Eb\Big[e^{-\frac{W_{\infty}}{bc_0}}\Big]}{\Eb\Big[1-e^{-\frac{W_{\infty}}{bc_0}}\Big]}.
\end{align*}
Besides 
\begin{align*}
    \E\Big[e^{-\sum_{j=1}^{N_{a,b}}f(\bm{c}_{\infty}c_0(b-a)\xi_j)}\Big|N_{a,b}>0\Big]=\frac{\E\Big[e^{-\sum_{j=1}^{N_{a,b}}f(\bm{c}_{\infty}c_0(b-a)\xi_j)}\Big]}{\P\big(N_{a,b}>0\big)}-\frac{\P\big(N_{a,b}=0\big)}{\P\big(N_{a,b}>0\big)},
\end{align*}
and recalling that $N_{a,b}=\sum_{i=1}^{N_b}G^i_{a,b}$, where $N_b$ is, under $\P^{\mathcal{E}}$, a Poisson random variable with parameter $W_{\infty}/(bc_0)$, $(G^i_{a,b})_{i\geq 1}$ is a sequence of\textrm{ i.i.d }Geometric random variables on $\N^*$ with probability of success $1-a/b$. $(\xi_j)_{j\geq 1}$ is a sequence of\textrm{ i.i.d }Exponential random variables with mean $1$, we clearly have that $\P(N_{a,b}>0)=\P(N_b>0)=\Eb[1-e^{-\frac{W_{\infty}}{bc_0}}]$ and
\begin{align*}
    \E^{\mathcal{E}}\Big[e^{-\sum_{j=1}^{N_{a,b}}f(\bm{c}_{\infty}c_0(b-a)\xi_j)}\Big]=\E^{\mathcal{E}}\Big[\phi_{a,b}(f)^{N_{a,b}}\Big]=\E^{\mathcal{E}}\Big[\Big(\frac{(b-a)\phi_{a,b}(f)}{b-a\phi_{a,b}(f)}\Big)^{N_b}\Big]&=e^{-\frac{W_{\infty}}{bc_0}\big(1-\frac{(b-a)\phi_{a,b}(f)}{b-a\phi_{a,b}(f)}\big)} \\ & =e^{-\frac{W_{\infty}(1-\phi_{a,b}(f))}{c_0(b-a\phi_{a,b}(f))}},
\end{align*}
and this completes the proof of \eqref{ConvDoubleRange2}. We can now conclude. By definition (see \eqref{DefSampling}) and thanks to \eqref{ConvDoubleRange1}, we have
\begin{align*}
    \P\big(|\VectCoord{\mathcal{X}}{1,n}_{\bn}\land\VectCoord{\mathcal{X}}{2,n}_{\bn}|\geq\an\big)=\E\left[\frac{\mathcal{C}^{(n)}_{\bn,\an}}{\big( R^{(n)}_{\bn}\big)^2}\Big|\; R^{(n)}_{\bn}>0\right]\underset{n\to\infty}{\longrightarrow}\E\left[\frac{\sum_{j=1}^{N_{a,b}}(\xi_j)^2}{(\sum_{j=1}^{N_{a,b}}\xi_j)^2}\Big| N_{a,b}>0\right],
\end{align*}
which finally gives \eqref{RecentPastCoa}. \\
For the alternative expression \eqref{ExpressionSerie}, we know thanks to section 3.1 in \cite{HarrisJohnstonRoberts1} that for any integer $\ell\geq 1$, $\E^{\mathcal{E}}[\sum_{j=1}^{\ell}(\xi_j)^2/(\sum_{j=1}^{\ell}\xi_j)^2]=2/(\ell+1)$. Moreover, one can prove that for any $\ell\geq 1$, $\P^{\mathcal{E}}(N_{a,b}=\ell)=\sum_{j=1}^{\ell}\frac{1}{j!}(W_{\infty}/(bc_0))^je^{-\frac{W_{\infty}}{bc_0}}\binom{\ell-1}{j-1}a^{\ell-j}(b-a)^j/b^{\ell}$, thus giving the result.\\
We now prove \eqref{RemotePastCoa}, that is, for any $m\in\N^*$ 
\begin{align*}
    \lim_{n\to\infty}\P\big(|\VectCoord{\mathcal{X}}{1,n}_{\bn}\land\VectCoord{\mathcal{X}}{2,n}_{\bn}|<m\big) =1-\E\left[\frac{\sum_{|u|=m}\Big(Z^{(u)}_{\infty}\mathcal{Y}^{(u)}_{b/Z^{(u)}_{\infty}}\Big)^2}{\Big(\sum_{|u|=m}Z^{(u)}_{\infty}\mathcal{Y}^{(u)}_{b/Z^{(u)}_{\infty}}\Big)^2}\Bigg|\max_{|u|=m}\mathcal{Y}^{(u)}_{b/Z^{(u)}_{\infty}}>0\right],
\end{align*}
where we recall that for almost-every environment $\mathcal{E}$, $(\mathcal{Y}^{(z)};\; |z|=m)$ are\textrm{ i.i.d }copies under $\P^{\mathcal{E}}$ of $\mathcal{Y}$, $(W^{(z)}_{\infty};\; |z|=m)$ are\textrm{ i.i.d }copies of $W_{\infty}$ and independent of $((z,V(z));\; |z|=m)$ and $Z^{(u)}_{\infty}=e^{-V(u)}W^{(u)}_{\infty}$. We first show that, in law under $\P$ 
\begin{align}\label{ConvDoubleRange}
    &\Big(\Big(\frac{1}{n^2}\mathcal{C}^{(n)}_{\bn,m},\; \frac{1}{n}R^{(n)}_{\bn}\Big)\Big|\; R^{(n)}_{\bn}>0\Big)\underset{n\to\infty}{\longrightarrow}\nonumber \\[0.7em] & \Big(\Big(\sum_{|u|=m}\Big(\bm{c}_{\infty}Z^{(u)}_{\infty}\mathcal{Y}^{(u)}_{b/Z^{(u)}_{\infty}}\Big)^2,\;\sum_{|u|=m}\bm{c}_{\infty}Z^{(u)}_{\infty}\mathcal{Y}^{(u)}_{b/Z^{(u)}_{\infty}}\Big)\Big|\; \max_{|u|=m}\mathcal{Y}^{(u)}_{b/Z^{(u)}_{\infty}}>0\Big),
\end{align}
where we recall that for any $p\geq 1$ and $1\leq k'\leq k$ 
\begin{align*}
    \mathcal{C}^{(p)}_{k,k'}=\sum_{|x|=|y|=k}\un_{\{N_x^{(p)}\land N_y^{(p)}\geq 1,\; |x\land y|\geq k'\}},
\end{align*}
which is introduced in the proof of \eqref{RecentPastCoa}. The convergence in \eqref{ConvDoubleRange} immediately gives \eqref{RemotePastCoa}. Let $f:\R\to\R$ be a continuous, non-negative and monotonic function such that $f(0)=0$ and recall that
\begin{align*}
    \mathbfcal{C}^{(n,p)}_{k,k'}(f)=\sum_{|u|=k'}f\Big(\frac{1}{n}\underset{x>u}{\sum_{|x|=k}}\un_{\{N_x^{(p)}\geq 1\}}\Big)\;\textrm{ and }\;\mathbfcal{C}^{(n)}_{k,k'}(f)=\mathbfcal{C}^{(n,n)}_{k,k'}(f),
\end{align*}
see below equation \eqref{ConvDoubleRange1}. We have
\begin{align}\label{ConvLoiCoaRoot}
    \lim_{n\to\infty}\E\Big[e^{-\mathbfcal{C}^{(n)}_{m,\bn}(f)}\Big|\; R^{(n)}_{\bn}>0\Big]=\E\left[e^{-\sum_{|u|=m}f\big(Z^{(u)}_{\infty}\mathcal{Y}^{(u)}_{b/Z^{(u)}_{\infty}}\big)}\Big|\; \max_{|u|=m}\mathcal{Y}^{(u)}_{b/Z^{(u)}_{\infty}}>0\right].
\end{align}
Indeed, we have
\begin{align*}
    \E\Big[e^{-\mathbfcal{C}^{(n)}_{m,\bn}(f)}\Big|\;R^{(n)}_{\bn}>0\Big]=\frac{\E\Big[e^{-\mathbfcal{C}^{(n)}_{m,\bn}(f)}\Big]}{\P\big(R^{(n)}_{\bn}>0\big)}-\frac{\P\big(R^{(n)}_{\bn}=0\big)}{\P\big(R^{(n)}_{\bn}>0\big)}.
\end{align*}
Thanks to the branching property (see Fact \ref{FactGWMulti})
\begin{align*}
    \E\Big[e^{-\mathbfcal{C}^{(n)}_{m,\bn}(f)}\Big]=\E\Big[\prod_{u\in\mathcal{R}^{n};|u|=m}g_{\bn-m}\big(n,N_u^{(n)},f\big)\Big].
\end{align*}
Thanks to Theorem \ref{ThMultiMar}, we have, for any $t\geq 0$ that $\lim_{n\to\infty}g_{\bn-m}(n,\tn,f)=g(t,f)$ with $g(t,f)=\E[\exp(-f(\bm{c}_{\infty}\mathcal{Z}_b(tW_{\infty})))]$. Moreover, $t\mapsto g_{\bn-m}(n,\tn,f)$ is monotonic and $t\mapsto g(t,f)$ is continuous so Dini's Theorem yields $\lim_{n\to\infty}\sup_{t\in[0,M]}|g_{\bn-m}(n,\tn,f)-g(t,f)|=0$ for any $M>0$. One can see that for any $m\geq 1$, $\lim_{n\to\infty}\max_{|u|=m}|N^{(n)}_u/n-e^{-V(u)}|$ in $\P$-probability and recall that we have assumed (see section \ref{RBRWT}) that the set $\{u\in\T;\; |u|=m\}$ is finite for all $m\geq 1$, $\Pb$-almost surely. Therefore, dominated convergence theorem yields
\begin{align*}
    \lim_{n\to\infty}\E\Big[e^{-\mathbfcal{C}^{(n)}_{m,\bn}(f)}\Big]=\lim_{n\to\infty}\E\Big[\prod_{u\in\mathcal{R}^{n};|u|=m}g_{\bn-m}\big(n,N_u^{(n)},f\big)\Big]=\Eb\Big[\prod_{|u|=m}g\big(e^{-V(u)},f\big)\Big],
\end{align*}
which is nothing but $\E[\exp(-\sum_{|u|=m}f\big(Z^{(u)}_{\infty}\mathcal{Y}^{(u)}_{b/Z^{(u)}_{\infty}}))]$. Moreover, we have that $\P(R^{(n)}_{\bn}>0)\to\Eb[1-e^{-\frac{W_{\infty}}{bc_0}}]$ as $n\to\infty$, thus giving 
\begin{align*}
    \lim_{n\to\infty}\E\Big[e^{-\mathbfcal{C}^{(n)}_{m,\bn}(f)}\Big|\;R^{(n)}_{\bn}>0\Big]=&\frac{1}{\Eb\Big[1-e^{-\frac{W_{\infty}}{bc_0}}\Big]}\E\left[e^{-\sum_{|u|=m}f\big(Z^{(u)}_{\infty}\mathcal{Y}^{(u)}_{b/Z^{(u)}_{\infty}}\big)}\right]-\frac{1-\Eb\Big[e^{-\frac{W_{\infty}}{bc_0}}\Big]}{\Eb\Big[1-e^{-\frac{W_{\infty}}{bc_0}}\Big]}.
\end{align*}
Finally, note that 
\begin{align*}
    \P^{\mathcal{E}}\Big(\max_{|u|=m}\mathcal{Y}^{(u)}_{b/Z^{(u)}_{\infty}}>0\Big)=1-\prod_{|u|=m}\P^{\mathcal{E}}\Big(\mathcal{Y}_{b/Z^{(u)}_{\infty}}=0\Big)=1-e^{-\frac{1}{bc_0}\sum_{|u|=m}Z^{(u)}_{\infty}},
\end{align*}
and since $\sum_{|u|=m}Z^{(u)}_{\infty}$ and $W_{\infty}$ have the same distribution under $\Pb$ (recall that $W_{\infty}$ satisfies \eqref{SmoothingTransform}), the latter limit is noting but $\E[\exp(-\sum_{|u|=m}f\big(Z^{(u)}_{\infty}\mathcal{Y}^{(u)}_{b/Z^{(u)}_{\infty}}))|\; \max_{|u|=m}\mathcal{Y}^{(u)}_{b/Z^{(u)}_{\infty}}>0]$ and \eqref{ConvLoiCoaRoot} is proved. To get $\eqref{ConvDoubleRange}$, one can see that, on the event $\{R^{n}_{\bn}>0\}$
\begin{align*}
    R^{(n)}_{\bn}=\sum_{|u|=m}\;\underset{x>u}{\sum_{|x|=\bn}}\un_{\{N_x^{(n)}\geq 1\}}\;\;\textrm{ and }\;\;\mathcal{C}^{(n)}_{\bn,m}=\sum_{|u|=m}\Big(\underset{x>u}{\sum_{|x|=\bn}}\un_{\{N_x^{(n)}\geq 1\}}\Big)^2.
\end{align*}
Hence, taking $f(\mathfrak{z})=\lambda_1\mathfrak{z}+\lambda_2\mathfrak{z}^2$ with $\lambda_1,\lambda_2\geq 0$ in \eqref{ConvLoiCoaRoot} gives the convergence in law \eqref{ConvDoubleRange}. \\
For the alternative expression \eqref{RemotePastCoa2}, one can see that
\begin{align*}
    &\E\left[\frac{\sum_{|u|=m}\Big(Z^{(u)}_{\infty}\mathcal{Y}^{(u)}_{b/Z^{(u)}_{\infty}}\Big)^2}{\Big(\sum_{|u|=m}Z^{(u)}_{\infty}\mathcal{Y}^{(u)}_{b/Z^{(u)}_{\infty}}\Big)^2}\Bigg|\max_{|u|=m}\mathcal{Y}^{(u)}_{b/Z^{(u)}_{\infty}}>0\right] \\ & =\frac{1}{1-\Eb\Big[e^{-\frac{W_{\infty}}{bc_0}}\Big]}\Eb\left[\sum_{|v|=m}\E^{\mathcal{E}}\left[\frac{\Big(Z^{(v)}_{\infty}\mathcal{Y}^{(v)}_{b/Z^{(v)}_{\infty}}\Big)^2}{\Big(\sum_{|u|=m}Z^{(u)}_{\infty}\mathcal{Y}^{(u)}_{b/Z^{(u)}_{\infty}}\Big)^2}\un_{\big\{\mathcal{Y}^{(v)}_{b/Z^{(v)}_{\infty}}>0\big\}}\right]\right].
\end{align*}
Using that $\frac{1}{t^2}=\int_0^{\infty}\lambda e^{-\lambda t}\mathrm{d}\lambda$, we have, for any $v\in\T$ such that $|v|=m$
\begin{align*}
    \E^{\mathcal{E}}\left[\frac{\Big(Z^{(v)}_{\infty}\mathcal{Y}^{(v)}_{b/Z^{(v)}_{\infty}}\Big)^2}{\Big(\sum_{|u|=m}Z^{(u)}_{\infty}\mathcal{Y}^{(u)}_{b/Z^{(u)}_{\infty}}\Big)^2}\un_{\big\{\mathcal{Y}^{(v)}_{b/Z^{(v)}_{\infty}}>0\big\}}\right]=\int_0^{\infty}\lambda\E^{\mathcal{E}}\left[\Big(Z^{(v)}_{\infty}\mathcal{Y}^{(v)}_{b/Z^{(v)}_{\infty}}\Big)^2e^{-\lambda\sum_{|u|=m}Z^{(u)}_{\infty}\mathcal{Y}^{(u)}_{b/Z^{(u)}_{\infty}}}\right]\mathrm{d}\lambda,
\end{align*}
and using that $(\mathcal{Y}^{(u)};\; |u|=m)$ is a collection of\textrm{ i.i.d }copies of $\mathcal{Y}$ (see \eqref{LaplaceCSBP}), we obtain
\begin{align*}
    \E^{\mathcal{E}}\left[\Big(Z^{(v)}_{\infty}\mathcal{Y}^{(v)}_{b/Z^{(v)}_{\infty}}\Big)^2e^{-\lambda\sum_{|u|=m}Z^{(u)}_{\infty}\mathcal{Y}^{(u)}_{b/Z^{(u)}_{\infty}}}\right]&=\E^{\mathcal{E}}\left[\Big(Z^{(v)}_{\infty}\mathcal{Y}^{(v)}_{b/Z^{(v)}_{\infty}}\Big)^2e^{-\lambda Z^{(v)}_{\infty}\mathcal{Y}^{(v)}_{b/Z^{(v)}_{\infty}}}\right]\underset{u\not=v}{\prod_{|u|=m}}\E^{\mathcal{E}}\left[e^{-Z^{(u)}_{\infty}\mathcal{Y}^{(u)}_{b/Z^{(u)}_{\infty}}}\right] \\ & =\frac{(Z^{(v)}_{\infty})^2+2bc_0(1+\lambda bc_0)Z^{(v)}_{\infty}}{(1+\lambda bc_0)^4}e^{-\frac{\lambda}{1+\lambda bc_0}Z^{(v)}_{\infty}}\underset{u\not=v}{\prod_{|u|=m}}e^{-\frac{\lambda}{1+\lambda bc_0}Z^{(u)}_{\infty}} \\ & =\frac{(Z^{(v)}_{\infty})^2+2bc_0(1+\lambda bc_0)Z^{(v)}_{\infty}}{(1+\lambda bc_0)^4}e^{-\frac{\lambda}{1+\lambda bc_0}\sum_{|u|=m}Z^{(u)}_{\infty}},
\end{align*}
which yields what we wanted. We are left to prove \eqref{SommeVaut1}, that is
\begin{align*}
    \lim_{a\to 0}\;\lim_{n\to\infty}\P(|\VectCoord{\mathcal{X}}{1,n}_{\bn}\land\VectCoord{\mathcal{X}}{2,n}_{\bn}|\geq\an)+\lim_{m\to\infty}\;\lim_{n\to\infty}\P(|\VectCoord{\mathcal{X}}{1,n}_{\bn}\land\VectCoord{\mathcal{X}}{2,n}_{\bn}|<m)=1.
\end{align*}
For the first limit, recall that $N_{a,b}=\sum_{i=1}^{N_b}G^i_{a,b}$, where $N_b$ is, under $\P^{\mathcal{E}}$, a Poisson random variable with parameter $W_{\infty}/(bc_0)$, $(G^i_{a,b})_{i\geq 1}$ is a sequence of\textrm{ i.i.d }Geometric random variables on $\N^*$ with probability of success $1-a/b$. In particular, $G^i_{a,b}\to 1$ as $a\to0$, in law under $\P^{\mathcal{E}}$. By independence, $N_{a,b}\to N_b$ as $a\to 0$, in law under $\P^{\mathcal{E}}$ and 
\begin{align*}
    \lim_{a\to 0}\;\lim_{n\to\infty}\P(|\VectCoord{\mathcal{X}}{1,n}_{\bn}\land\VectCoord{\mathcal{X}}{2,n}_{\bn}|\geq\an)&=\lim_{a\to\infty}\E\left[\frac{\sum_{j=1}^{N_{a,b}}\big(\xi_{j}\big)^2}{\big(\sum_{j=1}^{N_{a,b}}\xi_{j}\big)^2}\Big|N_{a,b}>0\right] \\[0.7em] & =\E\left[\frac{\sum_{j=1}^{N_b}\big(\xi_{j}\big)^2}{\big(\sum_{j=1}^{N_b}\xi_{j}\big)^2}\Big|N_b>0\right].
\end{align*}
Let us now prove that 
\begin{align*}
    \lim_{m\to\infty}\E\left[\frac{\sum_{|u|=m}\Big(Z^{(u)}_{\infty}\mathcal{Y}^{(u)}_{b/Z^{(u)}_{\infty}}\Big)^2}{\Big(\sum_{|u|=m}Z^{(u)}_{\infty}\mathcal{Y}^{(u)}_{b/Z^{(u)}_{\infty}}\Big)^2}\Bigg|\max_{|u|=m}\mathcal{Y}^{(u)}_{b/Z^{(u)}_{\infty}}>0\right]=\E\left[\frac{\sum_{j=1}^{N_b}\big(\xi_{j}\big)^2}{\big(\sum_{j=1}^{N_b}\xi_{\bn}\big)^2}\Big|N_b>0\right].
\end{align*}
For that, we show that, in law, under $\P$
\begin{align}\label{CVlaw_m}
    \Big(\sum_{|u|=m}\Big(Z^{(u)}_{\infty}\mathcal{Y}^{(u)}_{b/Z^{(u)}_{\infty}}\Big)^2,\; \sum_{|u|=m}Z^{(u)}_{\infty}\mathcal{Y}^{(u)}_{b/Z^{(u)}_{\infty}}\Big)\underset{m\to\infty}{\longrightarrow}\Big(\sum_{j=1}^{N_{b}}\big(\bm{c}_{\infty}c_0b\xi_{j}\big)^2,\; \sum_{j=1}^{N_{b}}\bm{c}_{\infty}c_0b\xi_{j}\Big).
\end{align}
Let $\lambda_1,\lambda_2$ and $f(\mathfrak{z})=\lambda_1\mathfrak{z}^2+\lambda_2\mathfrak{z}$. As in the proof of \eqref{ConvLoiCoaRoot}
\begin{align*}
    \E\left[e^{-\sum_{|u|=m}f\Big(Z^{(u)}_{\infty}\mathcal{Y}^{(u)}_{b/Z^{(u)}_{\infty}}\Big)}\right]=\E\Big[\prod_{|u|=m}g\big(e^{-V(u)},f\big)\Big],
\end{align*}
where we recall that $g(t,f)=\E[\exp(-f(\bm{c}_{\infty}\mathcal{Z}_b(tW_{\infty})))]$. Note that under $\P^{\mathcal{E}}$, $\mathcal{Z}_b(tW_{\infty})$ is distributed as $c_0b\sum_{j=1}^{N_b(t)}\xi_j$ where $N_b(t)$ is a Poisson random variable with parameter $tW_{\infty}/(c_0b)$ and independent of $(\xi_j)_{j\geq 1}$. Hence, under $\P^{\mathcal{E}}$
\begin{align*}
    f\big(\bm{c}_{\infty}\mathcal{Z}_b(tW_{\infty})\big)\overset{\textrm{(law)}}{=} G_b(t)+\lambda_1(\bm{c}_{\infty}c_0b)^2F_b(t),
\end{align*}
where $G_b(t)=\lambda_1(\bm{c}_{\infty}c_0b)^2\sum_{j=1}^{N_b(t)}(\xi_{j})^2+\lambda_2\bm{c}_{\infty}c_0b\sum_{j=1}^{N_b(t)}\xi_{j}$ and $F_b(t):=\sum_{i\not=j=1}^{N_b(t)}\xi_i\xi_{j}$. Using that for any $m\geq 1$, $W_{\infty}$ is distributed as $\sum_{|u|=m}Z^{(u)}_{\infty}=e^{-V(u)}W_{\infty}^{(u)}$ under $\Pb$, one can see that
\begin{align*}
    \E\Big[e^{-\sum_{j=1}^{N_b}f(\xi_j)}\Big]=\Eb\Big[e^{\frac{W_{\infty}}{bc_0}\E^{\mathcal{E}}[1-e^{-f(\xi_1)}]}\Big]&=\Eb\Big[\prod_{|u|=m}e^{\frac{1}{bc_0}e^{-V(u)}W^{(u)}_{\infty}\E^{\mathcal{E}}[1-e^{-f(\xi_1)}]}\Big] \\ & =\Eb\Big[\prod_{|u|=m}e^{-G_b(e^{-V(u)})}\Big],
\end{align*}
so in order to get \eqref{CVlaw_m}, we only have to prove that 
\begin{align*}
    \lim_{m\to\infty}\E\Big[\sum_{|u|=m}F_b\big(e^{-V(u)}\big)\Big]=0,
\end{align*}
We have $\E^{\mathcal{E}}[F_b(t)]=\E^{\mathcal{E}}[N_b(t)(N_b(t)-1)]=(tW_{\infty})^2/(c_0b)^2$ thus giving $\E[F_b(t)]=t^2/(c_0b)^2$, where we have used that $\Eb[(W_{\infty})^2]=1$ since $\kappa>2$. Hence
\begin{align*}
    \E\Big[\sum_{|u|=m}F_b\big(e^{-V(u)}\big)\Big]=\frac{1}{(c_0b)^2}\Eb\Big[\sum_{|u|=m}e^{-2V(u)}\Big]=\frac{e^{m\psi(2)}}{(c_0b)^2}\underset{m\to\infty}{\longrightarrow}0.
\end{align*}
Let us finally recall that for all $m\geq 1$, $\P(\max_{|u|=m}\mathcal{Y}^{(u)}_{b/Z^{(u)}_{\infty}}>0)=\P(N_b>0)$ and the proof is completed.
\end{proof}

\subsection{Proofs of \eqref{RemSmallGenerations21}, \eqref{RemSmallGenerations22} and \eqref{RemarkGen1Excu}}\label{LastSection}

We first prove in this section the continuity between small generations and critical generations evoked in Remark \ref{RemSmallGenerations2}, that is equation \eqref{RemSmallGenerations21}
\begin{align*}
    \lim_{b\to 0}\;\underset{a<b}{\lim_{a\to 0}}\;\lim_{n\to\infty}\P\big(|\VectCoord{\mathcal{X}}{1,n}_{\bn}\land\VectCoord{\mathcal{X}}{2,n}_{\bn}|\geq\an\big)=0,
\end{align*}
and equation \eqref{RemSmallGenerations22}, that is for any $m\geq 1$
\begin{align*}
    \lim_{b\to 0}\;\lim_{n\to\infty}\P\big(|\VectCoord{\mathcal{X}}{1,n}_{\bn}\land\VectCoord{\mathcal{X}}{2,n}_{\bn}|<m\big)=\Eb^*\left[\frac{\mathcal{A}_{n,m}}{(W_{\infty})^2}\right],
\end{align*}
where we recall that $\mathcal{A}_{n,m}=\lim_{n\to\infty}\sum_{x\not= y;|x|=|y|=n}e^{-V(x)}e^{-V(y)}\un_{\{|x\land y|<m\}}$, see \eqref{RemSmallGenerations21} and \eqref{RemSmallGenerations22}. Indeed, we have proved (see the proof of Theorem \ref{GenealogyDiffCritic})
\begin{align*}
    \underset{a<b}{\lim_{a\to 0}}\;\lim_{n\to\infty}\P\big(|\VectCoord{\mathcal{X}}{1,n}_{\bn}\land\VectCoord{\mathcal{X}}{2,n}_{\bn}|\geq\an\big)=\E\left[\frac{\sum_{j=1}^{N_b}\big(\xi_{j}\big)^2}{\big(\sum_{j=1}^{N_b}\xi_{j}\big)^2}\Big|N_b>0\right],
\end{align*}
where we recall that $N_b$ is, under $\P^{\mathcal{E}}$, a Poisson random variable with parameter $W_{\infty}/(bc_0)$ and $(\xi_j)_{j\geq 1}$ is a sequence of $\textrm{ i.i.d }$ Exponential random variables with mean $1$, this sequence being independent of $N_b$. Recall that for any $\ell\geq 1$, $\E^{\mathcal{E}}[\sum_{j=1}^{\ell}(\xi_j)^2/(\sum_{j=1}^{\ell}\xi_j)^2]=2/(\ell+1)$ so
\begin{align*}
    \E\left[\frac{\sum_{j=1}^{N_b}\big(\xi_{j}\big)^2}{\big(\sum_{j=1}^{N_b}\xi_{j}\big)^2}\Big|N_b>0\right]=\frac{2}{\P(N_b>0)}\E\Big[\un_{\{N_b>0\}}\big(N_b+1\big)^{-1}\Big]=\frac{2\E\big[\big(N_b+1\big)^{-1}\big]}{\Eb\big[1-e^{-\frac{W_{\infty}}{bc_0}}\big]}-\frac{2\Eb\big[e^{-\frac{W_{\infty}}{bc_0}}\big]}{\Eb\big[1-e^{-\frac{W_{\infty}}{bc_0}}\big]}.
\end{align*}
Note that $\E[(N_b+1)^{-1}]=\Pb(W_{\infty}=0)+bc_0\Eb[\un_{\{W_{\infty}>0\}}(1-e^{-\frac{W_{\infty}}{bc_0}})/W_{\infty}]\to\Pb(W_{\infty}=0)>0$ as $b\to\infty$. Moreover, $\Eb[e^{-\frac{W_{\infty}}{bc_0}}]\to\Pb(W_{\infty}=0)$ as $b\to 0$, thus giving \eqref{RemSmallGenerations21}. \\
For \eqref{RemSmallGenerations22}, we have thanks to Theorem \ref{GenealogyDiffCritic}
\begin{align*}
    \lim_{n\to\infty}\P\big(|\VectCoord{\mathcal{X}}{1,n}_{\bn}\land\VectCoord{\mathcal{X}}{2,n}_{\bn}|<m\big) =1-\E\left[\frac{\sum_{|u|=m}\Big(Z^{(u)}_{\infty}\mathcal{Y}^{(u)}_{b/Z^{(u)}_{\infty}}\Big)^2}{\Big(\sum_{|u|=m}Z^{(u)}_{\infty}\mathcal{Y}^{(u)}_{b/Z^{(u)}_{\infty}}\Big)^2}\Bigg|\max_{|u|=m}\mathcal{Y}^{(u)}_{b/Z^{(u)}_{\infty}}>0\right].
\end{align*}
Recall that $(\mathcal{Y}_a;\; a\geq 0)$ has continuous paths and $\mathcal{Y}_0=1$ so for any continuous function $f:\R\to\R$, $\P^{\mathcal{E}}$-almost surely, $\lim_{b\to0}\sum_{|u|=m}f(Z^{(u)}_{\infty}\mathcal{Y}^{(u)}_{b/Z^{(u)}_{\infty}})=\sum_{|u|=m}f(Z^{(u)}_{\infty})$. In particular, for any $\lambda_1,\lambda_2\geq 0$
\begin{align*}
    &\E\Big[e^{-\lambda_1\sum_{|u|=m}\big(Z^{(u)}_{\infty}\mathcal{Y}^{(u)}_{b/Z^{(u)}_{\infty}}\big)^2}e^{-\lambda_2\sum_{|u|=m}Z^{(u)}_{\infty}\mathcal{Y}^{(u)}_{b/Z^{(u)}_{\infty}}}\Big|\max_{|u|=m}\mathcal{Y}^{(u)}_{b/Z^{(u)}_{\infty}}>0\Big] \\ & =\frac{\E\Big[e^{-\lambda_1\sum_{|u|=m}\big(Z^{(u)}_{\infty}\mathcal{Y}^{(u)}_{b/Z^{(u)}_{\infty}}\big)^2}e^{-\lambda_2\sum_{|u|=m}Z^{(u)}_{\infty}\mathcal{Y}^{(u)}_{b/Z^{(u)}_{\infty}}}\Big]}{\P\big(\max_{|u|=m}\mathcal{Y}^{(u)}_{b/Z^{(u)}_{\infty}}>0\big)}-\frac{\P\big(\max_{|u|=m}\mathcal{Y}^{(u)}_{b/Z^{(u)}_{\infty}}=0\big)}{\P\big(\max_{|u|=m}\mathcal{Y}^{(u)}_{b/Z^{(u)}_{\infty}}>0\big)} \\ & \underset{b\to0}{\longrightarrow}\frac{\Eb\Big[e^{-\lambda_1\sum_{|u|=m}(Z^{(u)}_{\infty})^2}e^{-\lambda_2\sum_{|u|=m}Z^{(u)}_{\infty}}\Big]}{\Pb(W_{\infty}>0)}-\frac{\Pb(W_{\infty}=0)}{\Pb(W_{\infty}>0)},
\end{align*}
where we have also used that $\P(\max_{|u|=m}\mathcal{Y}^{(u)}_{b/Z^{(u)}_{\infty}}>0)=\Eb[1-\prod_{|u|=m}e^{-Z^{(u)}_{\infty}/(bc_0)}]\to \Pb(W_{\infty}>0)$ as $b\to0$ since $W_{\infty}$ satisfies \eqref{SmoothingTransform}. Hence, noting that $(Z^{(u)}_{\infty};\; |u|=m)$ is a collection of\textrm{ i.i.d }copies of $(\lim_{n\to\infty}\sum_{|x|=n;\; x>u}e^{-V(x)};\; |u|=m)$, we have
\begin{align*}
    &\lim_{b\to0}\E\Big[e^{-\lambda_1\sum_{|u|=m}\big(Z^{(u)}_{\infty}\mathcal{Y}^{(u)}_{b/Z^{(u)}_{\infty}}\big)^2}e^{-\lambda_2\sum_{|u|=m}Z^{(u)}_{\infty}\mathcal{Y}^{(u)}_{b/Z^{(u)}_{\infty}}}\Big|\max_{|u|=m}\mathcal{Y}^{(u)}_{b/Z^{(u)}_{\infty}}>0\Big] \\ & =\Eb\Big[e^{-\lambda_1\lim_{n\to\infty}\sum_{|u|=m}\big(\sum_{|x|=n;\; x>u}e^{-V(x)}\big)^2}e^{-\lambda_2W_{\infty}}\Big|W_{\infty}>0\Big].
\end{align*}
Recalling that that $\Pb$-almost surely, $\{W_{\infty}>0\}=\{ \textrm{non-extinction of }\T\}$, we finally obtain
\begin{align*}
    \lim_{b\to\infty}\lim_{n\to\infty}\P\big(|\VectCoord{\mathcal{X}}{1,n}_{\bn}\land\VectCoord{\mathcal{X}}{2,n}_{\bn}|<m\big) =\Eb^*\Big[\frac{\mathcal{A}_{n,m}}{(W_{\infty})^2}\Big],
\end{align*}
which is exactly what we wanted. \\
We end this section with a brief proof of \eqref{RemarkGen1Excu}, that is
\begin{align*}
    \lim_{n\to\infty}\P\big(|\VectCoord{\mathcal{X}}{1,1}_{\bn}\land\VectCoord{\mathcal{X}}{2,1}_{\bn}|\geq\an\big)=\E\left[\frac{\sum_{j=1}^{G_{a,b}}(\xi_j)^2}{(\sum_{j=1}^{N_{a,b}}\xi_j)^2}\right]=\frac{2(b-a)}{a^2}\Big(b\log\Big(\frac{b}{b-a}\Big)-a\Big).
\end{align*}
Again, by definition, it is enough to show that 
\begin{align*}
    \lim_{n\to\infty}\E\Big[e^{-\mathbfcal{C}^{(n,1)}_{\bn,\an}(f)}\Big|R^{(1)}_{\bn}>0\Big]=\E\Big[e^{-\sum_{j=1}^{G_{a,b}}f(\bm{c}_{\infty}c_0(b-a)\xi_j)}\Big],
\end{align*}
where we recall that, $\mathbfcal{C}^{(n,p)}_{\bn,\an}(f)=\sum_{|u|=\an}f(\frac{1}{n}\sum_{|x|=\bn;x>u}\un_{\{N_x^{(p)}\geq 1\}})$ and $f(\mathfrak{z})=\lambda_1\mathfrak{z}^2+\lambda_2\mathfrak{z}$ with $\lambda_1,\lambda_2\geq 0$.  Indeed
\begin{align*}
    \E\Big[e^{-\mathbfcal{C}^{(n,1)}_{\bn,\an}(f)}\Big|R^{(1)}_{\bn}>0\Big]=\frac{\P\big(R^{(1)}_{\an}>0\big)}{\P\big(R^{(1)}_{\bn}>0\big)}\left(\E\Big[e^{-\mathbfcal{C}^{(n,1)}_{\bn,\an}(f)}\Big|R^{(1)}_{\an}>0\Big]-\P\big(R^{(1)}_{\bn}=0\big|R^{(1)}_{\an}>0\big)\right)
\end{align*}
As in the proof of \eqref{ConvDoubleRange2}, using the branching property (see Fact \ref{FactGWMulti}), Lemma \ref{ConvLoiCondi} and that $n\P(N^{(1)}_{\an}>0)\to1/(ac_0)$ as $n\to\infty$, on can prove that
\begin{align*}
    \lim_{n\to\infty}\E\Big[e^{-\mathbfcal{C}^{(n,1)}_{\bn,\an}(f)}\Big|R^{(1)}_{\an}>0\Big]=\lim_{n\to\infty}\E\left[e^{-\frac{(1-\phi_{a,b}(f))N^{(1)}_{\an}}{nc_0(b-a)}}\Bigg|N^{(1)}_{\an}>0\right]=\frac{b-a}{b-a\phi_{a,b}(f)},
\end{align*}
and also using that $R^{(1)}_{\bn}=0$ if and only if $\sum_{|x|=\bn;\;x>u}\un_{\{N^{(1)}_x\geq 1\}}=0$ for all $u\in\mathcal{R}^{(1)}$ such that $|u|=\an$, we have
\begin{align*}
    \lim_{n\to\infty}\P\big(R^{(1)}_{\bn}=0\big|R^{(1)}_{\an}>0\big)=\lim_{n\to\infty}\E\Big[\P\big(R^{(1)}_{\bn-\an}=0\big)^{N^{(1)}_{\an}}\big|N^{(1)}_{\an}>0\Big]=\frac{b-a}{b}.
\end{align*}
Hence
\begin{align*}
    \lim_{n\to\infty}\E\Big[e^{-\mathbfcal{C}^{(n,1)}_{\bn,\an}(f)}\Big|R^{(1)}_{\bn}>0\Big]=\frac{b}{a}\Big(\frac{b-a}{b-a\phi_{a,b}(f)}-\frac{b-a}{a}\Big)=\frac{(b-a)\phi_{a,b}(f)}{b-a\phi_{a,b}(f)},
\end{align*}
which is nothing but $\E[e^{-\sum_{j=1}^{G_{a,b}}f(\bm{c}_{\infty}c_0(b-a)\xi_j)}]$, thus giving the first equality in \eqref{RemarkGen1Excu}. For the second one, recall that for any $\ell\geq 1$, $\E[\sum_{j=1}^{\ell}(\xi_j)^2/(\sum_{j=1}^{\ell}\xi_j)^2]=2/(\ell+1)$ so
\begin{align*}
    \E\left[\frac{\sum_{j=1}^{G_{a,b}}(\xi_j)^2}{(\sum_{j=1}^{N_{a,b}}\xi_j)^2}\right]=\frac{2b(b-a)}{a^2}\sum_{\ell\geq 1}\frac{1}{\ell+1}\Big(\frac{a}{b}\Big)^{\ell+1}=\frac{2b(b-a)}{a^2}\int_0^{a/b}\frac{s}{1-s}\mathrm{d}s=\frac{2(b-a)}{a^2}\Big(b\log\Big(\frac{b}{b-a}\Big)-a\Big),
\end{align*}
and this ends the proof of \eqref{RemarkGen1Excu}.

\vspace{1cm}

\noindent\begin{merci}
The author is grateful to Pierre Andreoletti for interesting discussions in the early stages. The author is partially supported for this work by the NZ Royal Society Te Apārangi Marsden Fund (22-UOA-052) entitled "Genealogies of samples of individuals selected at random from stochastic populations: probabilistic structure and applications". The author also benefited during the preparation of this paper from the support of Hua Loo-Keng Center for Mathematical Sciences (AMSS, Chinese Academy of Sciences) and the National Natural Science Foundation of China (No. 12288201).
    
\end{merci}

\bibliographystyle{alpha}
\bibliography{thbiblio}

\end{document}